\documentclass[a4paper]{article}%{AIAA}%

\usepackage{amsmath}
\usepackage{amssymb}
\usepackage{graphicx}
\usepackage{color}
\usepackage{multirow}
\usepackage{subfigure}
\graphicspath{{./figure/}}
\usepackage{setspace}
\usepackage{latexsym}
\usepackage{amsfonts}
\usepackage{amsthm}
\usepackage{indentfirst}
\usepackage[numbers,sort&compress]{natbib}
\usepackage[colorlinks,
            linkcolor=blue,
            anchorcolor=blue,
            citecolor=blue
            ]{hyperref}

\usepackage{geometry}
\def\mybibliography{
   \list
   {$[{\arabic{enumi}}]$}{\settowidth\labelwidth{[99]}
   \leftmargin\labelwidth \advance\leftmargin\labelsep
   \usecounter{enumi}}
   \def\newblock{\hskip .11em plus .07em minus -.07em}
   \sloppy \sfcode`\.=1000\relax }

\def\p{\boldsymbol{p}}
\def\q{\boldsymbol{q}}

\def\x{\boldsymbol{x}}
\def\y{\boldsymbol{y}}

\def\f{\boldsymbol{f}}
\def\F{\boldsymbol{F}}

\def\r{\boldsymbol{r}}

\def\v{\boldsymbol{v}}

\def\u{\boldsymbol{u}}

\def\1{\boldsymbol{1}}

\def\0{\boldsymbol{0_3}}

\begin{document}

\title{Optimality Conditions Applied to Free-Time Multi-Burn Optimal Orbital Transfers}

\author{\it
Zheng Chen\thanks{Laboratoire de Math\'ematiques d'Orsay, Univ. Paris-Sud, CNRS, Universit\'e Paris-Saclay, 91405, France. \underline{zheng.chen@math.u-psud.fr.}},\\
%{\it Universit\'e Paris-Sud, Orsay, Ile-de-France, France, 91405}
}

\maketitle{}

\newcommand{\eqnref}[1]{(\ref{#1})}
 \newcommand{\class}[1]{\texttt{#1}}
 \newcommand{\package}[1]{\texttt{#1}}
 \newcommand{\file}[1]{\texttt{#1}}
 \newcommand{\BibTeX}{\textsc{Bib}\TeX}
 \newcommand{\upcite}[1]{{\textsuperscript{\cite{#1}}}}

\newtheorem{definition}{Definition}%[section]
\newtheorem{proposition}{Proposition}%[section]
\newtheorem{problem}{Problem}%[section]
\newtheorem{remark}{Remark}%[section]
\newtheorem{assumption}{Assumption}%[section]
\newtheorem{hypothesis}{Hypothesis}%[section]
\newtheorem{conjecture}{Conjecture}%[section]
\newtheorem{theorem}{Theorem}%[section]
\newtheorem{corollary}{Corollary}%[section]
\newtheorem{lemma}{Lemma}%[section]
\newtheorem{condition}{Condition}

\begin{abstract}
While the Pontryagin Maximum Principle can be used to calculate candidate extremals for optimal orbital transfer problems, these candidates cannot be guaranteed  to be at least locally optimal unless sufficient optimality conditions are satisfied. In this paper, through constructing a parameterized family of extremals around a reference extremal, some second-order necessary and sufficient conditions for the strong-local optimality of the free-time multi-burn fuel-optimal transfer  are established under certain regularity assumptions. Moreover,  the numerical procedure for computing these optimality conditions is presented. Finally, two medium-thrust fuel-optimal trajectories with different number of burn arcs for a typical orbital transfer problem are computed and the local optimality of the two computed trajectories are tested thanks to the second-order optimality conditions established in this paper.
\end{abstract}

\section*{Nomenclature}
\noindent\begin{tabular}{@{}lcl@{}}
$\f$, $\f_0$, $\f_1$ &=& vector fields\\
$h$, $H$ &=& Hamiltonian and maximized Hamiltonian\\
$I_n$ &=& identity matrix of $\mathbb{R}^{n\times n}$\\
$m$, $m_c$ &=& mass and the mass of a spacecraft without any fuel, kg\\
$\mathbb{N}$ &=& set of natural numbers\\
$Oxyz$ &=& an Earth-centered inertial Cartesian coordinate\\
$\p$, $\p_r$, $\p_v$, $p_m$ &=& costates\\
$\r$ &=& position vector, m\\
$\mathbb{R}^n$, $(\mathbb{R}^n)^*$ &=& spaces of $n$-dimensional column and row vectors\\
$t$ &=& time, s\\
$\u$, $\mathcal{U}$ &=& control (or thrust) vector and its admissible set\\
$u_{max}$ &=& maximum magnitude of thrust\\
$\v$ &=& velocity vector, m/s\\
$\x$, $\mathcal{X}$ &=& state vector and its admissible set\\
$\boldsymbol{0}_{i\times j}$ &=& zero matrix of $\mathbb{R}^{i\times j}$\\
$\mu$ &=& Earth gravitational parameter\\
$\Pi$ &=& canonical projection\\

subscript\\
$f$ &=& final condition\\
$0$ &=& initial condition\\

superscript\\

$T$ &=& matrix transpose
\end{tabular} \\

\section{Introduction}

\indent Since the Pontryagin Maximum Principle (PMP) \cite{Pontryagin} was developed by a group of Russian researchers in 1950s, an increasing number of papers on the topic of space trajectory optimization have been published, showing that the PMP is a powerful tool to identify candidate extremals for optimal orbital transfer problems. However, the PMP requiring the first variation of a cost functional to vanish cannot guarantee these candidates to be at least locally optimal unless second-order necessary and sufficient optimality conditions are satisfied. Moreover, the satisfaction of sufficient conditions is a prerequisite to perform a neighboring optimal feedback guidance scheme \cite{Kelley:62,Chuang:96,Chuang:95,Lee:65}. Hence, once a candidate extremal is computed by the PMP, it is indeed crucial to establish sufficient optimality conditions which, when met, guarantee that the candidate is at least locally optimal.

The classical Jacobi no-conjugate-point condition, derived from the calculus of variations \cite{Bryson:69,Mermau:76,Breakwell:65}, has been widely used to test second-order necessary and sufficient conditions if the extremal is smooth. The test  is generally done by checking the explosive time of the matrix solution of a Riccati differential equation. Based on this method, the second-order sufficient conditions for singular (or intermediate-thrust) space trajectories, along which the PMP is trivially satisfied, were studied in \cite{Kelley:66,Popescu:06,Breakwell:75}. Using a transition matrix method, which transforms the test of the unboundedness of the matrix solution into detecting the zero of a scalar, a numerical procedure for testing the sufficient optimality conditions for continuous-thrust orbital transfer trajectories has been developed recently in \cite{Jo:00,Prussing:05}. Nevertheless, a challenge arizes when we consider a finite-thrust fuel-optimal problem because the corresponding optimal control function   exhibits a bang-bang behavior if the transfer time is greater than the minimum transfer time for the same boundary conditions \cite{Gergaud:06}. To the author's knowledge, through testing conjugate points on each burn arc, Chuang {\it et al.} \cite{Chuang:96,Chuang:95} presented a primary study on the sufficient optimality conditions for planar multi-burn orbital transfer problems.

Second order conditions in the bang-bang case have received an extensive treatment; references include the paper of Sarychev \cite{Sarychev:97} followed by \cite{Agrachev:02} and \cite{Maurer:04,Osmolovskii:05,Osmolovskii:07}.  More recently, a regularization procedure has been developed in \cite{Silva:10} for single-input systems. These papers consider controls taking values in polyhedra, showing that conjugate points occur only at switching times. However, the control for the orbital transfer problem studied in the present paper takes values in a Euclidean ball. In recent years, a study on the method of characteristics carried out by Noble and Sch\"attler \cite{Noble:02} shows that along a bang-bang extremal conjugate points can occur not only on each smooth bang arc but also at a switching point if a transversality condition at the switching point is violated (see a more recent work in \cite{Schattler:12}). Assuming the endpoints of an optimal control problem are fixed, it has been proven in \cite{Caillau:15} that a bang-bang extremal realizes a strict strong-local optimum if both the Jacobi no-conjugate-point condition and the transversality condition are satisfied on each smooth bang arc and at each switching point, respectively. Generalizing the problem with fixed endpoints to the problem that the final point varies on a smooth submanifold, an extra necessary and sufficient condition, involving the geometry of the final constraint manifold, has been established in \cite{Chen:153bp} recently.  However, as is shown in Sect. \ref{SE:Sufficient}, one cannot  apply the optimality conditions developed in \cite{Schattler:12,Noble:02,Caillau:15,Chen:153bp} to the free-time optimal orbital transfer problem. In this paper,  through employing the geometric methods developed  in  \cite{Agrachev:04,Schattler:12,Noble:02,Caillau:15,Bonnard:07,Chen:153bp}, the sufficient optimality conditions for the free-time multi-burn orbital transfer problem are established and the numerical procedure for testing such conditions is presented. The crucial idea is to construct a parameterized family of extremals around a reference extremal such that the theory of field of extremals can be applied.

The paper is organized as follows. In Sect. \ref{SE:Problem_Formulation}, the finite-thrust fuel-optimal orbital transfer problem is formulated, and the first-order necessary conditions are derived by applying the PMP. In Sect. \ref{SE:Sufficient},
 under some regularity assumptions, three second-order sufficient conditions, ensuring a bang-bang extremal trajectory of the free-time orbital transfer problem to be a strict strong-local optimum, are established. In Sect. \ref{SE:Procedure}, a numerical implementation for these sufficient conditions is presented. In Sect. \ref{SE:Numerical}, to illustrate the theoretical development of this paper, two fuel-optimal  trajectories with different number of burn arcs for a typical orbital transfer problem  are calculated.

\section{Optimal control problem}\label{SE:Problem_Formulation}

\subsection{Dynamics}

%\subsubsection{Dynamics}

Consider the spacecraft as a mass point moving around the Earth. The state  in an  Earth-centered inertial Cartesian coordinate, denoted by $Oxyz$, consists of the position vector $\r\in\mathbb{R}^3\backslash\{0\}$, the velocity vector $\v\in\mathbb{R}^3$, and the mass $m\in\mathbb{R}_+$. Let $t\in\mathbb{R}_+$ be the time, then the differential equations for the finite-thrust orbital transfer problem can be written as
\begin{eqnarray}
\begin{cases}
\dot{\r}(t) = \v(t),\\
\dot{\v}(t) = -\frac{\mu}{\parallel \r(t) \parallel^3}\r(t) + \frac{\boldsymbol{u}(t)}{m(t)},\\
\dot{m}(t) = -\beta{\parallel \boldsymbol{u}(t)\parallel},\label{EQ:mass}
\end{cases}
\label{EQ:Sigma}
\end{eqnarray}
where $\|\cdot\|$ denotes the Euclidean norm, $\beta > 0$ is a scalar constant determined by the specific impulse of the engine equipped on the spacecraft. The thrust (or control) vector $\boldsymbol{u}\in\mathbb{R}^3$   takes values in the admissible set
\begin{equation}
\mathcal{U}=\big\{\boldsymbol{u} \in \mathbb{R}^3 \ \arrowvert\ \parallel \boldsymbol{u} \parallel \leq u_{max}\big\},\nonumber\\
\end{equation}
where $u_{max}>0$ is the maximum magnitude of thrust.
 Let $n=7$ be the dimension of the state space and denote by $\x\in\mathbb{R}^n$ the state such that $\x = (\r,\v,m)$; we define the admissible set for state $\x$ by
\begin{eqnarray}
\mathcal{X}=\{(\r,\v,m)\in\mathbb{R}^3\backslash\{0\} \times \mathbb{R}^3\times\mathbb{R}_+\ \arrowvert\  \r\times\v \neq 0,\  m \geq m_c\},\nonumber
\end{eqnarray}
where $m_c>0$ is the mass of the spacecraft without any fuel.

Let $\rho\in[0,1]$ be the normalized mass flow rate of the engine, i.e., $\rho = \parallel \boldsymbol{u}\parallel/u_{max}$, and let $\boldsymbol{\omega}\in\mathbb{S}^2$ be the unit vector of thrust direction.
 In order to avoid heavy notations, we define the controlled vector field $\f$ on $\mathcal{X}\times\mathcal{U}$ by
\begin{eqnarray}
\f:\mathcal{X}\times\mathcal{U}\rightarrow T_{\x}\mathcal{X},\ \f(\x,\boldsymbol{u}) = \f_0(\x) + \rho\f_1(\x,\boldsymbol{\omega}),\nonumber
\end{eqnarray}
where 
\begin{eqnarray}
\f_0(\x) = \left(\begin{array}{c}\v\\
-\frac{\mu}{\parallel \r\parallel^3}\r\\
0
\end{array}\right)\ \text{and}\ \f_1(\x,\boldsymbol{\omega}) = \left(\begin{array}{c}
\boldsymbol{0}\\
\frac{u_{max}}{m}\boldsymbol{\omega}
\\-\beta{u_{max}}
\end{array}\right).\nonumber
\end{eqnarray}
Then, the dynamics  in Eq.(\ref{EQ:Sigma}) can be rewritten as
\begin{eqnarray}
\dot{\x}(t) =\f(\x(t),\boldsymbol{u}(t)) = \f_0(\x(t)) + \rho(t) \f_1(\x(t),\boldsymbol{\omega}(t)).
\label{EQ:system}
\end{eqnarray}
This form of dynamics will be used later to establish sufficient optimality conditions.

\subsection{Fuel-optimal problem}\label{Subsection:Fuel}

Let  $\x_f\in\mathcal{X}$ be the final state and let $s\in\mathbb{N}$ be a positive integer such that $0<s\leq n$; we define the constraint submanifold of the final state $\x_f$ by
\begin{eqnarray}
 \mathcal{M} = \big\{\x_f\in\mathcal{X}\ \arrowvert\ \phi(\x_f) = 0\big\},\label{EQ:final_manifold}
\end{eqnarray}
where $\phi:\mathcal{X}\rightarrow \mathbb{R}^{s}$ is a twice continuously differentiable function of $\x_f$ and its expression depends on specific mission requirements. Then, the fuel-optimal problem is defined as following.
 \begin{definition}[Fuel-optimal problem ({\it FOP})]
Given a fixed initial point $\x_0\in\mathcal{X}\backslash\mathcal{M}$, the fuel-optimal problem consists of steering the system of Eq. (\ref{EQ:system}) in the admissible set $\mathcal{X}$ by a measurable control $\boldsymbol{u}(\cdot)\in\mathcal{U}$ on a finite time interval $[0,t_f]$  from the initial point $\x_0$  to a final point $\x_f\in\mathcal{M}$ such that the fuel consumption is minimized, i.e., 
\begin{eqnarray}
\int_{0}^{t_f}\rho(t) dt \rightarrow \text{min},
\label{EQ:cost_functional}
\end{eqnarray}
where $t_f>0$ is the free final time.
\end{definition}
\noindent It is worth remarking here that either the number of burn arcs or the final true longitude\footnote{The true longitude is the sum of the true anomaly, the argument of periapsis, and the argument of right ascending node of the classical orbital elements (see \cite{Broucke:72} for detailed definition).} has to been fixed when solving the free-time orbital transfer problem; otherwise the problem is ill-posed \cite{Oberle:97}. The controllability of the system in Eq. (\ref{EQ:system}) holds in the admissible set $\mathcal{X}$ for every positive $u_{max}$ if $m_c$ is small enough (see, e.g., \cite{Chen:15controllability}). Let $t_m>0$ be the minimum transfer time from the initial point $\x_0$ to a final point $\x_f\in\mathcal{M}$, if $t_f \geq t_m$, there exists at least one fuel-optimal solution in $\mathcal{X}$ according to the existence result of Gergaud and Haberkorn \cite{Gergaud:06}. Thanks to the controllability and the existence results, the PMP is applicable to formulate the following Hamiltonian system. 

\subsubsection{Hamiltonian system}

According to the PMP \cite{Pontryagin}, if an admissible controlled trajectory $\x(\cdot)\in\mathcal{X}$ associated with a measurable control $\boldsymbol{u}(\cdot)\in\mathcal{U}$  on $[0,t_f]$ is an optimal one of the {\it FOP}, there exists a nonpositive real number $p^0$ and an absolutely continuous mapping $t\mapsto\p(\cdot)\in T_{\x(\cdot)}^*\mathcal{X}$ on $[0,t_f]$, satisfying $(\p,p^0)\neq 0$ and called adjoint state, such that almost everywhere on $[0,t_f]$ there holds
\begin{eqnarray}
\begin{cases}
\dot{\x}(t) = \frac{\partial h}{\partial \p}(\x(t),\p(t),p^0,\boldsymbol{u}(t)),\\
\dot{\p}(t) = -\frac{\partial h}{\partial \x}(\x(t),\p(t),p^0,\boldsymbol{u}(t)),
\end{cases}
\label{EQ:cannonical}
\end{eqnarray}
and
\begin{eqnarray}
h({\x}(t),{\p}(t),{p}^0,\boldsymbol{u}(t)) =\underset{\boldsymbol{u}^*(t)\in\mathcal{U}}{\text{max}} h({\x}(t),{\p}(t),{p}^0,\boldsymbol{u}^*(t)) ,
\label{EQ:maximum_condition}
\end{eqnarray}
where 
\begin{eqnarray}
h(\x,\p,p^0,\boldsymbol{u}) = \p\f_0(\x) + \rho \p\f_1(\x,\boldsymbol{\omega}) + p^0 \rho
\label{EQ:Hamiltonian}
\end{eqnarray}
 is the Hamiltonian. Since the final time is free and the dynamics is not dependent on time explicitly, there holds
 \begin{eqnarray}
 h(\x(t),\p(t),p^0,\u(t))\equiv 0, \ t\in[0,{t}_f].
 \label{EQ:Hamiltonian_=0}
 \end{eqnarray}
Moreover,  the  boundary transversality condition implies
\begin{eqnarray}
\boldsymbol{p}(t_f) =  \boldsymbol{\nu}{{\nabla} \phi(\x(t_f))},%\\
\label{EQ:Transversality}
\end{eqnarray}
where the notation ``~$\nabla$~'' denotes the vector differential operator and $\boldsymbol{\nu}\in(\mathbb{R}^{s})^*$ is a constant row vector whose elements are Lagrangian multipliers.

The 4-tuple $t\mapsto (\x(t),\p(t),p^0,\boldsymbol{u}(t))\in T^*\mathcal{X}\times \mathbb{R}\times\mathcal{U}$ on $[0,t_f]$, if satisfying Eqs.~(\ref{EQ:cannonical}--\ref{EQ:Hamiltonian}), is called an extremal.  Furthermore, an extremal is called a normal one if $p^0\neq 0$ and it is called an abnormal one if $p^0 = 0$. The abnormal extremals have been ruled out by Gergaud and Haberkorn \cite{Gergaud:06}. Thus, only normal extremals are considered and $(\p,p^0)$ is normalized such that $p^0 = -1$ hereafter. According to the maximum condition in Eq.~(\ref{EQ:maximum_condition}), given every extremal $(\x(\cdot),\p(\cdot),p^0,\boldsymbol{u}(\cdot))$ on $[0,t_f]$, the corresponding extremal control  $\boldsymbol{u}(\cdot)$ is a function of $(\x(\cdot),\p(\cdot))$ on $[0,t_f]$, i.e., $\boldsymbol{u}(\cdot) = \boldsymbol{u}(\x(\cdot),\p(\cdot))$ on $[0,t_f]$. Thus, with some abuses of notations, we denote by $(\x(\cdot),\p(\cdot))\in T^*\mathcal{X}$ on $[0,t_f]$ the normal extremal and $H(\x(\cdot),\p(\cdot))$ on $[0,t_f]$ the corresponding maximized Hamiltonian, i.e., 
$$H(\x(t),\p(t)) := \underset{\boldsymbol{\u}^*(t)\in \mathcal{U}}{\text{max}}\ h(\x(t),\p(t),-1,\boldsymbol{u}^*(t)),\ t\in[0,t_f],$$
 which  is rewritten as
\begin{eqnarray}
H(\x(t),\p(t)) = H_0(\x(t),\p(t)) + \rho(\x(t),\p(t)) H_1(\x(t),\p(t)),\nonumber
\end{eqnarray}
where $
H_0(\x,\p) = \p \f_0(\x)$ is the drift Hamiltonian and $
H_1(\x,\p) = \p \f_1(\x,\boldsymbol{\omega}(\x,\p)) - 1$ is the switching function.

\subsubsection{Necessary Conditions}\label{Subsection:Necessary}

Let $\p_r\in T_{\r}^*\mathbb{R}^3$, $\p_v\in T_{\v}^*\mathbb{R}^3$, and $p_m\in T_{m}^*\mathbb{R}$ be the costates with respect to $\r$, $\v$, and $m$, respectively, such that $\p = (\p_r,\p_v,p_m)$. Then the maximum condition in Eq.~(\ref{EQ:maximum_condition}) implies 
\begin{eqnarray}
\boldsymbol{\omega} = \p_v^T / \|\p_v \|,\ \  \text{if}\ \ \| \p_v\|\neq 0,
\label{EQ:Max_condition2}
\end{eqnarray}
and
\begin{eqnarray}
\begin{cases}
\rho =1,\ \ \ \ \  \text{if}\ H_1 > 0,
\\
\rho = 0, \ \ \ \ \ \text{if}\ H_1 < 0.
\end{cases}
\label{EQ:Max_condition1}
\end{eqnarray}
Thus, the optimal direction of the thrust vector $\boldsymbol{u}$ is collinear with the adjoint vector $\p_v$ which is well-known as  the primer vector \cite{Lawden:63}. 
While an extremal $(\x(\cdot),\p(\cdot))\in T^*\mathcal{X}$ on $[0,t_f]$ is called a nonsingular one if $H_1(\x(\cdot),\p(\cdot))$ has only isolated zeros on $[0,t_f]$, it is called  a singular one if there is a finite interval $[t_1,t_2]\subseteq[0,t_f]$ such that $H_1(\x(\cdot),\p(\cdot))\equiv 0$ on $[t_1,t_2]$.

Though the necessary conditions in Eqs.~(\ref{EQ:cannonical}-\ref{EQ:Transversality}) can be used to compute extremals by solving a two-point boundary value problem \cite{Pan:13}, the computed extremals cannot be guaranteed to be at least locally optimal unless sufficient optimality conditions are satisfied. Assuming an extremal is totally singular, the sufficient conditions have been studied by  Breakwell {\it et al.} \cite{Breakwell:75}  and  Popescu \cite{Popescu:06} independently. For nonsingular extremals with totally continuous thrust, e.g., the extremals of time-optimal orbital transfer problems, both the procedures developed in \cite{Prussing:05,Jo:00} and the classical methods in \cite{Wood:74,Mermau:76,Breakwell:65,Bryson:69} can be directly used to test sufficient optimality conditions. In next section, the sufficient conditions for the strong-local optimality of the nonsingular extremals with bang-bang controls will be established.

\section{Sufficient optimality conditions for bang-bang extremals}\label{SE:Sufficient}

 Before studying the sufficient optimality conditions, we firstly give the following definition of local optimality \cite{Bonnard:07}.
\begin{definition}
Given an extremal trajectory $\bar{\x}(\cdot)\in \mathcal{X}$ of the {\it FOP}, let $\bar{t}_f>0$ be the optimal final time and let $\bar{\boldsymbol{u}}(\cdot)\in\mathcal{U}$ on $[0,\bar{t}_f]$ be the extremal control. Then, assuming $\sigma>0$ is small enough, we say that $\bar{\x}(\cdot)\in \mathcal{X}$  on $[0,\bar{t}_f]$ realizes a weak-local optimum in $L^{\infty}$-topology (resp. strong-local optimum in $C^0$-topology) if there exists an open neighborhood $\mathcal{W}_{\boldsymbol{u}}\subseteq\mathcal{U}$ of $\bar{\boldsymbol{u}}(\cdot)$ in $L^{\infty}$-topology (resp. an open neighborhood $\mathcal{W}_{\x}\subseteq\mathcal{X}$ of $\bar{\x}(\cdot)$ in $C^{0}$-topology) such that for every $t_f\in[\bar{t}_f-\sigma,\bar{t}_f+\sigma]$ and every admissible controlled trajectory $\x(\cdot)\in\mathcal{X}$ associated with the measurable control $\boldsymbol{u}(\cdot)\in\mathcal{W}_{\boldsymbol{u}}$ on $[0,t_f]$  (resp. every admissible controlled  trajectory $\x(\cdot)\in\mathcal{W}_{\x}$ associated with the measurable control $\boldsymbol{u}(\cdot)\in\mathcal{U}$ on $[0,t_f]$)  with the boundary conditions $\x(0) = \bar{\x}(0)$ and $\phi(\x(t_f))=0$, there holds 
$$\int_{0}^{t_f}\parallel \boldsymbol{u}(t) \parallel dt \geq \int_{0}^{\bar{t}_f}\parallel \bar{\boldsymbol{u}}(t) \parallel dt.$$
We say it  realizes a strict weak-local (resp. strict strong-local) optimum if the strict inequality holds.
\label{DE:Optimality}
\end{definition}
\noindent Note that, if a trajectory $\bar{\x}(\cdot)\in\mathcal{X}$ on $[0,t_f]$ realizes a strong-local optimum,  it automatically realizes a weak-local one.

\subsection{Parameterized family of extremals }

%Given a reference extremal $(\bar{\x}(\cdot),\bar{\p}(\cdot))\in T^*\mathcal{X}$ on $[0,t_f]$, let $\mathcal{P}\subset T^*_{\x_0}\mathcal{X}$ be an open neighbourhood of $\bar{\p}_0$. 
For every $\p_0\in T^*_{\x_0}\mathcal{X}$ and every $t_f>0$, we define by
\begin{eqnarray}
{\gamma}:[0,t_f]\times T^*_{\x_0}\mathcal{X}\rightarrow T^*\mathcal{X}, \ (t,\p_0)\mapsto (\x(t),\p(t)),\nonumber
\end{eqnarray}
the solution trajectory of Eqs.~(\ref{EQ:cannonical}--\ref{EQ:Hamiltonian}) such that $(\x_0,\p_0) = \gamma(0,\p_0)$.  In the remainder part of this paper, we specify $\bar{\p}_0\in T^*_{\x_0}\mathcal{X}$ and $\bar{t}_f\in\mathbb{R}_+$ in such a way that $\gamma(\cdot,\bar{\p}_0)$ on $[0,\bar{t}_f]$ is the extremal of the {\it FOP}. Hence, denoting by $\gamma(\cdot,\bar{\p}_0)$ on $[0,\bar{t}_f]$ the reference extremal, we will establish sufficient optimality conditions for this reference extremal hereafter.
\begin{definition}
Given the reference extremal $\gamma(\cdot,\bar{\p}_0)$ on $[0,\bar{t}_f]$, let $\mathcal{P}\subset T^*_{\x_0}\mathcal{X}$ be an open neighbourhood of $\bar{\p}_0$ and let $\sigma>0$ be small enough. Then, we define by
\begin{eqnarray}
\mathcal{F}_{\p_0} = \Big\{(\x(t),\p(t))\in T^*\mathcal{X}\ \arrowvert \ (\x(t),\p(t)) = \gamma(t,\p_0),\ t\in[0,t_f],\ t_f\in[\bar{t}_f-\sigma,\bar{t}_f+\sigma],\ \p_0\in\mathcal{P}\Big\},\nonumber
\end{eqnarray}
the $\p_0$-parameterized family of extremals around the reference extremal $\gamma(\cdot,\bar{\p}_0)$ on $[0,\bar{t}_f]$.
\end{definition}
\noindent Let us define by the mapping
\begin{eqnarray}
\Pi:  T^*\mathcal{X} \rightarrow \mathcal{X},\ \  \Pi(\x,\p) = \x,\nonumber
\end{eqnarray}
 the canonical projection that projects a submanifold from the cotangent bundle $T^*\mathcal{X}$ onto the state space $\mathcal{X}$.
If the restriction of $\Pi(\mathcal{F}_{\p_0})$ onto the state space $\mathcal{X}$ loses its local diffeomorphism at a time $t_c\in(0,t_f]$, we say the projection at  $t_c$ is a fold singularity. 
%\end{definition}

The local optimality of the reference extremal is related to  fold singularities of $\Pi(\mathcal{F}_{\p_0})$  through the notion of conjugate and focal point (see, e.g., \cite{Agrachev:04,Bonnard:07}), as is shown by the typical picture in Fig. \ref{Fig:smooth_fold}. 
\begin{figure}[!ht]
 \begin{center}
 \includegraphics[trim = 2cm 2.0cm 2cm 0cm, clip=true, width=3in, angle=0]{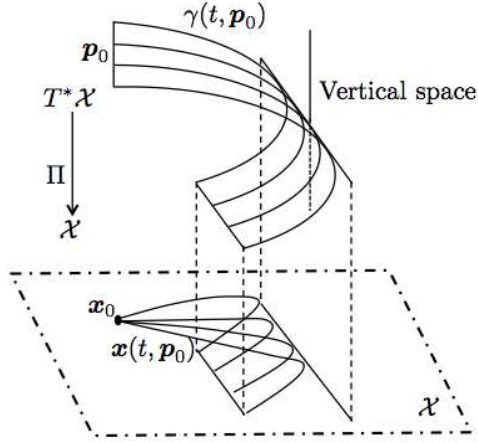}
 \end{center}
\caption[]{A typical picture for the occurrence of a conjugate point, i.e., the fold singularity for the  projection of $\mathcal{F}$ onto state space $\mathcal{X}$.}
 \label{Fig:smooth_fold}
\end{figure}
 Let $(\x(t,\p_0),\p(t,\p_0)) = \gamma(t,\p_0)$ for $(t,\p_0)\in[0,t_f]\times\mathcal{P}$ and assume that the final time  is fixed; a fold singularity occurs at a time $t_c\in(0,\bar{t}_f]$ if $\det\left[{\partial \x}(t_c,\bar{\p}_0)/{\partial \p_0}\right] = 0$ \cite{Schattler:12}. Hence, conjugate points for fixed-time orbital transfer problems are tested by detecting the zero of $\det\left[{\partial \x}(\cdot,\bar{\p}_0)/{\partial \p_0}\right]$ on $(0,\bar{t}_f]$ in \cite{Chen:153bp,Caillau:15,Agrachev:04}. 
  However, according to Eq.~(\ref{EQ:Hamiltonian_=0}), $H(\x,\p)= 0$ for every candidate extremal $(\x,\p)$ of the free-time problem. Thus, for the free-time problem, $\p_0$ lies in the subset
\begin{eqnarray}
\mathcal{H} = \{\p_0\in T^*_{\x_0}\mathcal{X}\ \arrowvert\ H(\x_0,\p_0) = 0\}.\nonumber
\end{eqnarray}
Note that $\bar{\p}_0\in\mathcal{H}$.
Since the subset $\mathcal{H}$ is locally diffeomorphic to $\mathbb{R}^{n-1}$, there holds $$\left.\text{rank}\left[\frac{\partial \x}{\partial \p_0}(\cdot,{\p}_0)\right]\right\rvert_{\p_0 = \bar{\p}_0} \leq n-1$$ everywhere on $[0,\bar{t}_f]$, which in further indicates $\det\left[{\partial \x}(\cdot,\bar{\p}_0)/{\partial \p_0}\right] \equiv 0$ on $[0,\bar{t}_f]$ if the final time is free \cite{Bonnard:07}. 
Therefore, for the free-time orbital transfer problem, one cannot test conjugate points by detecting the zero of $\det\left[{\partial \x}(\cdot,\bar{\p}_0)/{\partial \p_0}\right]$ any more. In next paragraph, a new parameterized family of extremals will be constructed such that the numerically verifiable conditions for conjugate points can be established.
 
 \begin{assumption}
 Given the reference extremal $(\bar{\x}(\cdot),\bar{\p}(\cdot)):=\gamma(\cdot,\bar{\p}_0)$ on $[0,\bar{t}_f]$, assume  the Hamiltonian $H(\bar{\x}(\cdot),\bar{\p}(\cdot))$ is regular on $[0,\bar{t}_f]$, i.e., ${\partial H}(\bar{\x}(\cdot),\bar{\p}(\cdot))/{\partial \x} \neq \boldsymbol{0}$ and ${\partial H}(\bar{\x}(\cdot),\bar{\p}(\cdot))/{\partial \p} \neq \boldsymbol{0}$ on $[0,\bar{t}_f]$.
 \label{AS:Regular_Hamiltonian}
 \end{assumption}
 \noindent As a result of this assumption,  there exists a full rank matrix $\boldsymbol{E}\in\mathbb{R}^{n\times(n-1)}$ such that its each column vector is 
orthogonal to the vector ${\partial H}(\x_0,\bar{\p}_0)/{\partial \p_0} = \f(\x_0,\u(\x_0,\bar{\p}_0))$.  Since the matrix $\boldsymbol{E}$ is of full rank, we are able to define an invertible function $\boldsymbol{F}:\mathcal{P}\cap\mathcal{H}\rightarrow (\mathbb{R}^{n-1})^*,\ \p_0\mapsto \F(\p_0)$ as
\begin{eqnarray}
\F(\p_0) :=   (\p_0-\bar{\p}_0) \boldsymbol{E}\left[\boldsymbol{E}^T \boldsymbol{E}\right]^{-1},
\label{EQ:q}
\end{eqnarray}
such that both the function and its inverse are smooth. For notational simplicity, given every  neighbourhood $\mathcal{P}$ of $\bar{\p}_0$, we define by
\begin{eqnarray}
\mathcal{Q} = \{\q\in(\mathbb{R}^{n-1})^*\ \arrowvert\ \q = \F(\p_0), \ \p_0\in\mathcal{P}\cap\mathcal{H}\},
\label{EQ:Q}
\end{eqnarray}
the subset associated with $\mathcal{P}$. If $\bar{\q} := \F(\bar{\p}_0)$, then there holds $\bar{\q} = \boldsymbol{0}$ and $\mathcal{Q}\subset(\mathbb{R}^{n-1})^*$ is an open neighborhood of $\bar{\q}$. For every $\q\in\mathcal{Q}$ and every $t_f>0$, we define by 
\begin{eqnarray}
\Gamma:[0,t_f]\times\mathcal{Q}\rightarrow T^*\mathcal{X},\ (t,\q)\mapsto (\x(t),\p(t)),
\label{EQ:Gamma}
\end{eqnarray}
the solution trajectory of  Eqs.~(\ref{EQ:cannonical}--\ref{EQ:Hamiltonian}) such that $(\x_0,\F^{-1}(\q)) = \Gamma(0,\q)$. It is clear that $\gamma(\cdot,\bar{\p}_0)=\Gamma(\cdot,\bar{\q})$ on $[0,\bar{t}_f]$.
\begin{definition}
Given the reference extremal $\Gamma(\cdot,\bar{\q})$ on $[0,\bar{t}_f]$, let $\sigma>0$ be small enough. Then, we define by
\begin{eqnarray}
\mathcal{F}_{\q} = \big\{(\x(t),\p(t))\in T^*\mathcal{X}\ \arrowvert\ (\x(t),\p(t)) = \Gamma(t,\q),\ t\in[0,t_f],\ t_f\in[\bar{t}_f-\sigma,\bar{t}_f+\sigma],\ \q\in\mathcal{Q}\big\},\nonumber
\end{eqnarray}
the $\q$-parameterized family of extremals around the reference extremal.
\end{definition}
\noindent According to Eqs.~(\ref{EQ:q}--\ref{EQ:Gamma}), there holds $\mathcal{F}_{\p_0} = \mathcal{F}_{\q}$ if $\p_0\in\mathcal{P}\cap\mathcal{H}$. Thus, it suffices to study the projection behaviour  of the family $\mathcal{F}_{\q}$ instead in order to formulate the conditions for conjugate points of the extremal $\Gamma(\cdot,\bar{\q})$ on $[0,\bar{t}_f]$.

\subsection{Sufficient optimality conditions for $s=n$}

Without loss of generality, let the positive integer $k\in\mathbb{N}$ be the number of switching times $\bar{t}_i$ ($i =1,2,\cdots,k$) along the extremal $\Gamma(\cdot,\bar{\q})$ on $[0,\bar{t}_f]$ such that $0 = \bar{t}_0< \bar{t}_1 < \bar{t}_2 < \cdots < \bar{t}_k < \bar{t}_{k+1} = \bar{t}_f$. 
\begin{assumption}
Along the extremal $(\bar{\x}(\cdot),\bar{\p}(\cdot)) = \Gamma(\cdot,\bar{\q})$ on $[0,\bar{t}_f]$, each switching point at the switching time $\bar{t}_i$  is assumed to be a regular one, i.e., $H_1(\bar{\x}(\bar{t}_i),\bar{\p}(\bar{t}_i)) = 0$ and ${\dot{H}_1}(\bar{\x}(\bar{t}_i),\bar{\p}(\bar{t}_i))\neq 0$ for $i=1,2,\cdots,k$.
\label{AS:Regular_Switching}
\end{assumption}
\noindent 
As a result of this assumption, if the subset $\mathcal{Q}$ is small enough, the number of switching times on every extremal $\Gamma(\cdot,\q)$ is $k$ and the $i$-th switching time is a smooth function of $\q$. Thus, we denote by 
\begin{eqnarray}
t_i:\mathcal{Q}\rightarrow \mathbb{R}_+,\ \q\mapsto t_i(\q),\nonumber
\end{eqnarray}
the $i$-th switching time of the extremals $\Gamma(\cdot,\q)$ in $\mathcal{F}_{\q}$. Let us denote  by $\delta(\cdot)$  the determinant of the matrix ${\nabla \x}(\cdot,\bar{\q})$ on $[0,\bar{t}_f]$, i.e.,
\begin{eqnarray}
\delta (t): = \det\left[{\nabla \x}({t,\bar{\q}})\right], \ t\in[0,\bar{t}_f],\nonumber
\end{eqnarray}
where ${\nabla \x}({t,\bar{\q}}) = \left(\dot{\x}(t,\bar{\q}),{\partial \x(t,\bar{\q})}/{\partial \q}\right)$. Note that $\delta(\cdot)$ on $[0,\bar{t}_f]$ is a piecewise continuous function (see, e.g., \cite{Caillau:15}).
\begin{remark}
Assuming the subset $\mathcal{Q}$ is small enough, the projection of the family $\mathcal{F}_{\q}$ restricted to each domain $(\bar{t}_i,\bar{t}_{i+1})\times\mathcal{Q}$ for $i=0,1,\cdots,k$ is a local diffeomorphism if $\delta(\cdot) \neq 0$ on $(\bar{t}_i,\bar{t}_{i+1})$ and the projection at a time $t_c\in(\bar{t}_i,\bar{t}_{i+1})$ is a fold singularity  if $\delta (t_c) = 0$ \cite{Bonnard:07}.
\end{remark}
\noindent  Therefore, one can test conjugate points for the free-time problem by detecting the zero of $\delta(\cdot)$ on $(\bar{t}_i,\bar{t}_{i+1})$ for $i=0,\ 1,\ \cdots,\ k$.
\begin{condition}
$\delta(\bar{t}_f) \neq 0$ and $\delta ({\cdot}) \neq 0$ on each open interval $(\bar{t}_i,\bar{t}_{i+1})$ for $i=0,1,\cdots,k$.
\label{AS:Disconjugacy_bang}
\end{condition}
\noindent Though this condition guarantees that the projection of the family $\mathcal{F}_{\q}$ restricted to each domain $(\bar{t}_{i},\bar{t}_{i+1})\times\mathcal{Q}$ for $i=0,1,\cdots,k$  is a diffeomorphism if the subset $\mathcal{Q}$ is small enough, it is not sufficient to guarantee that the projection of the subset $\mathcal{F}_{\q}$ restricted to the whole domain $(0,\bar{t}_f]\times\mathcal{Q}$ is a diffeomorphism as well, as Fig. \ref{Fig:trans} shows that the flows $\x(t,\q)$ may intersect with each other near the switching time $t_i(\q)$.
\begin{figure}[!ht]
 \centering\includegraphics[trim=2cm 1cm 2cm 0cm, clip=true, width=3in, angle=0]{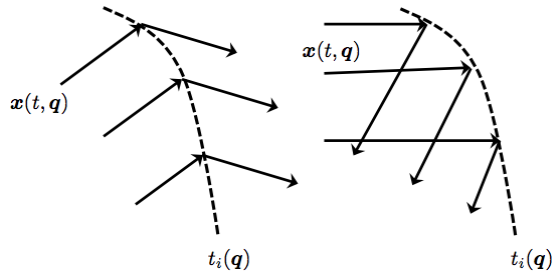}
\caption[]{The flows $\x(t,\q)\in\mathcal{X}$ near the switching time $t_i(\q)$. The left plot shows that the projection is a diffeomorphism, and the right plot shows that the projection is a fold singularity.}
 \label{Fig:trans}
\end{figure}
The behaviour of the fold singularity at switching times can be excluded by an appropriate transversality condition proposed in \cite{Noble:02}. In \cite{Caillau:15}, the transversality condition has been reduced to $\delta (\bar{t}_i-) \delta (\bar{t}_{i}+) > 0$ for  $i=1,2,\cdots,k$ where $\bar{t}_i-$ and $\bar{t}_i+$ denote the instants prior to and after the switching time $\bar{t}_i$, respectively. Moreover, it has been shown in \cite{Caillau:15} that the projection of the family $\mathcal{F}_{\q}$ near the switching time $t_i(\q)$ is a fold singularity if 
\begin{eqnarray}
\delta (\bar{t}_i-) \delta (\bar{t}_{i}+) < 0.
\label{EQ:fold_crossing}
\end{eqnarray}
Consequently, given the extremal $\Gamma(\cdot,\bar{\q})$ on $[0,\bar{t}_f]$, conjugate points may occur not only on a smooth bang arc if $\delta(t_c) = 0$ for a time $t_c\in(\bar{t}_{i},\bar{t}_{i+1})$ but also at a switching time $\bar{t}_i$ once  Eq.~(\ref{EQ:fold_crossing}) is satisfied. 

\begin{condition}
%\begin{eqnarray}
$\delta (\bar{t}_i-) \delta (\bar{t}_{i}+) > 0$ for  $i=1,2,\cdots,k$.
%\label{EQ:transversal_crossing}
%\end{eqnarray}
\label{AS:Transversality}
\end{condition}
\noindent  As is analyzed above, under  {\it Assumptions} \ref{AS:Regular_Hamiltonian} and \ref{AS:Regular_Switching}, the projection of the family $\mathcal{F}_{\q}$ restricted to the whole domain $(0,t_f]\times\mathcal{Q}$ is a diffeomorphism if the subset $\mathcal{Q}$ is small enough and if Conditions \ref{AS:Disconjugacy_bang} and \ref{AS:Transversality} are satisfied. Then, by directly applying the theory of field of extremals (cf. Proposition 17.2 and Theorem 17.2 in \cite{Agrachev:04}), one obtains the following result.
\begin{theorem}
Given the extremal $(\bar{\x}(\cdot),\bar{\p}(\cdot)) = \Gamma(\cdot,\bar{\q})$ on $[0,\bar{t}_f]$ such that  {\it Assumptions} \ref{AS:Regular_Hamiltonian} and \ref{AS:Regular_Switching} are satisfied, let $\sigma>0$ be small enough. Then, if Conditions \ref{AS:Disconjugacy_bang} and \ref{AS:Transversality} are satisfied and if the subset $\mathcal{Q}$ is small enough, every extremal trajectory ${\x}(\cdot,\q)=\Pi(\Gamma(\cdot,\q))$ associated with the extremal control $\u(\cdot,\q)\in\mathcal{U}$ on $[0,t_f]$ for $(t_f,\q)\in[\bar{t}_f-\sigma,\bar{t}_f+\sigma]\times\mathcal{Q}$ realizes a strict minimum cost in Eq.~(\ref{EQ:cost_functional}) among all the admissible controlled trajectories  $\x^*(\cdot)\in\Pi(\mathcal{F}_{{\q}})$ associated with the measurable control $\u^*(\cdot)\in\mathcal{U}$ on $[0,t^*_f]$ for $t^*_f>0$ with the same endpoints $\x^*(0) ={\x}(0,\q)$ and $\x^*(t^*_f) = {\x}({t}_f,\q)$, i.e., there holds
$$\int_0^{t_f} \|\u(t,\q)\| dt < \int_0^{t^*_f} \|\u^*(t) \|dt.$$%\nonumber
\label{CO:cor1}
\end{theorem}

Note that the final state $\x_f$ of the {\it FOP} is  fixed if $s=n$. Thus, in the case of $s=n$,  Theorem \ref{CO:cor1} indicates that  {\it Conditions} \ref{AS:Disconjugacy_bang}  and \ref{AS:Transversality}  are sufficient to guarantee the reference extremal $(\bar{\x}(\cdot),\bar{\p}(\cdot)) = \Gamma(\cdot,\bar{\q})$ on $[0,\bar{t}_f]$ to be a strict local optimum in the domain $\Pi(\mathcal{F}_{\q})$. Whereas, {\it Conditions} \ref{AS:Disconjugacy_bang}  and \ref{AS:Transversality}  are not sufficient any more if $s<n$ because one has to compare the cost of the reference extremal $\Gamma(\cdot,\bar{\q})$ on $[0,\bar{t}_f]$ with that of every admissible controlled trajectory $\x^*(\cdot)\in\Pi(\mathcal{F}_{\q})$ on $[0,t_f^*]$ not only with the same endpoints but also with the boundary conditions $\x^*(0)  = \bar{\x}(0)$ and $\x^*(t^*_f) \in\mathcal{M}\backslash\{\bar{\x}(\bar{t}_f)\}$ (see, e.g., \cite{Chen:153bp,Agrachev:98,Wood:74}). In next subsection, the sufficient conditions for the case of $s<n$ will be established.

\subsection{Sufficient optimality conditions for $s < n$}

Given a sufficiently small $\sigma>0$, let $\mathcal{N}\subset \mathcal{X}$ be the restriction of $\Pi(\mathcal{F}_{\q})$ on $[\bar{t}_f - \sigma,\bar{t}_f+\sigma]\times\mathcal{Q}$, i.e.,
\begin{eqnarray}
\mathcal{N} = \big\{\x\in\mathcal{X}\ \arrowvert\ \x = \Pi(\Gamma(t_f,\q)),\ t_f\in[\bar{t}_f-\sigma,\bar{t}_f+\sigma],\ \q\in\mathcal{Q}\big\}.\nonumber
\end{eqnarray}
If $\delta(\bar{t}_f)\neq 0$, the mapping $(t_f,\q)\mapsto \x(t_f,\q)$ on the domain $[\bar{t}_f-\sigma,\bar{t}_f+\sigma]\times\mathcal{Q}$ is a diffeomorphism. Thus, the subset $\mathcal{N}$ is an open neighborhood of $\bar{\x}(\bar{t}_f)$ if Conditions \ref{AS:Disconjugacy_bang} is satisfied, which in further implies $\mathcal{N}\cap\mathcal{M}\backslash\{\bar{\x}(\bar{t}_f)\}\neq \varnothing$ for the case of $s<n$.
\begin{definition}
Given the extremal $(\bar{\x}(\cdot),\bar{\p}(\cdot)) = \Gamma(\cdot,\bar{\q})$ on $[0,\bar{t}_f]$ and  a small $\varepsilon > 0$, if $\delta(\bar{t}_f)\neq 0$, we define by $$\y:[-\varepsilon,\varepsilon]\rightarrow\mathcal{M}\cap\mathcal{N},\ \xi\mapsto\y(\xi),$$ a twice continuously differentiable curve on $\mathcal{M}\cap\mathcal{N}$ such that $\y(0) = \bar{\x}(\bar{t}_f)$.
\end{definition}
\noindent Let us define by $\mathcal{O}\subseteq [\bar{t}_f -\sigma,\bar{t}_f + \sigma]\times\mathcal{Q}$ the subset of all $(t_f,\q)\in[\bar{t}_f -\sigma,\bar{t}_f + \sigma]\times\mathcal{Q}$ satisfying $\Pi(\Gamma(t_f,\q))\in\mathcal{M}\cap\mathcal{N}$, i.e.,
\begin{eqnarray}
\mathcal{O} = \big\{(t_f,\q)\in[\bar{t}_f - \sigma,\bar{t}_f + \sigma]\times\mathcal{Q}\ \arrowvert\ \Pi(\Gamma(t_f,\q))\in\mathcal{M}\cap\mathcal{N}\big\}.\nonumber
\end{eqnarray}

\begin{lemma}
Given the extremal $(\bar{\x}(\cdot),\bar{\p}(\cdot)) = \Gamma(\cdot,\bar{\q})$ on $[0,\bar{t}_f]$  such that {\it Assumptions} \ref{AS:Regular_Hamiltonian} and \ref{AS:Regular_Switching} are satisfied, let the subset $\mathcal{Q}$ and  $\sigma>0$ be small enough. Then, if Conditions \ref{AS:Disconjugacy_bang} and \ref{AS:Transversality} are satisfied, for every smooth curve $\y(\cdot)\in\mathcal{M}\cap\mathcal{N}$ on $[-\varepsilon,\varepsilon]$, there exists a smooth path $\xi\mapsto (\tau(\xi),\q(\xi))$ on $[-\varepsilon,\varepsilon]$ in $\mathcal{O}$ such that $\y(\xi) = \Pi(\Gamma(\tau(\xi),\q(\xi)))$.
\label{LE:lemma1}
\end{lemma}
\begin{proof}
As is analyzed previously, under the hypotheses of this lemma, the mapping $(t_f,\q)\mapsto\x(t_f,\q)$ restricted to the subset $\mathcal{O}$  is a diffeomorphism. Thus, according to the {\it inverse function theorem}, the lemma is proved. 
\end{proof}
\noindent Hereafter, we denote by $(\tau(\cdot),\q(\cdot))$ on $[-\varepsilon,\varepsilon]$ the smooth path on the subset $\mathcal{O}$ such that $\y(\xi) = \Pi(\Gamma(\tau(\xi),\q(\xi)))$ for $\xi\in[-\varepsilon,\varepsilon]$.  Let $\u(\cdot,\q(\xi))$ be the optimal control function of the extremal $\Gamma(\cdot,\q(\xi))$ on $[0,\tau(\xi)]$, and denote by $$J:[-\varepsilon,\varepsilon]\rightarrow \mathbb{R},\ \xi\mapsto J(\xi),$$ the cost functional of the extremal $\Gamma(\cdot,\q(\xi))$ on $[0,\tau(\xi)]$, i.e.,
\begin{eqnarray}
J(\xi) = \int_0^{\tau(\xi)}\| \u(t,\q(\xi))\| dt.
\label{EQ:definition_J_xi}
\end{eqnarray}
\begin{proposition}
Given the extremal $(\bar{\x}(\cdot),\bar{\p}(\cdot)) = \Gamma(\cdot,\bar{\q})$ on $[0,\bar{t}_f]$, let  {\it Assumptions} \ref{AS:Regular_Hamiltonian} and \ref{AS:Regular_Switching} be satisfied. Then, if Conditions \ref{AS:Disconjugacy_bang} and \ref{AS:Transversality} are satisfied, the extremal trajectory $\bar{\x}(\cdot)$ on $[0,\bar{t}_f]$ realizes a strict strong-local optimum (cf. Definition \ref{DE:Optimality}) if and only if there holds
\begin{eqnarray}
J(\xi) > J(0),\ \xi\in[-\varepsilon,\varepsilon]\backslash\{0\},
\label{EQ:lemma_1}
\end{eqnarray}
for every smooth curve $\y(\cdot)\in\mathcal{M}\cap\mathcal{N}$ on $[-\varepsilon,\varepsilon]$. 
\label{PR:compare}
\end{proposition}
\begin{proof}
Let us first prove that, under the hypotheses of this proposition, Eq.~(\ref{EQ:lemma_1}) is a sufficient condition for the strict strong-local optimality of the extremal trajectory $\bar{\x}(\cdot)$ on $[0,\bar{t}_f]$. Denote by $\x^*(\cdot)\in \Pi(\mathcal{F}_{\q})$  an admissible controlled trajectory associated with the measurable control $\u^*(\cdot)$ on  $[0,t^*_f]$ for $t^*_f\in[\bar{t}_f - \sigma,\bar{t}_f + \sigma]$ such that the boundary conditions $\x^*(0) = \x_0$ and $\x^*(t^*_f)\in\mathcal{M}\cap\mathcal{N}\backslash\{\bar{\x}(\bar{t}_f)\}$ are satisfied. Note that for every final point $\x^*(t_f^*)\in\mathcal{M}\cap\mathcal{N}\backslash\{\bar{\x}(\bar{t}_f)\}$, there must exist some smooth curves $\y(\cdot)\in\mathcal{M}\cap\mathcal{N}$ on $[-\varepsilon,\varepsilon]$ such that $\y({\xi}) = \x^*(t^*_f)$ for a ${\xi}\in[-\varepsilon,\varepsilon]\backslash\{0\}$. Thus, according to Lemma \ref{LE:lemma1}, there exist some correspondingly smooth paths $(\tau(\cdot),\q(\cdot))\in\mathcal{O}$ on $[-\varepsilon,\varepsilon]$ such that $\x_0 = \Pi(\Gamma(0,\q(\xi)))$ and $\x^*(t^*_f) = \Pi(\Gamma(\tau({\xi}),\q({\xi})))$. Then, according to Theorem \ref{CO:cor1}, there holds
\begin{eqnarray}
J({\xi})< \int_0^{t^*_f} \|\u^*(t)\|dt.\nonumber
%\label{EQ:lemma_2}
\end{eqnarray}
Substituting this equation into Eq.~(\ref{EQ:lemma_1}) implies
\begin{eqnarray}
J(0)< \int_0^{t^*_f} \|\u^*(t)\|dt,
\end{eqnarray}
i.e., the extremal trajectory $\bar{\x}(\cdot)$ on $[0,\bar{t}_f]$ realizes a strict minimum cost among every admissible controlled trajectory $\x^*(\cdot)\in\Pi(\mathcal{F}_{\q})$ on $[0,t^*_f]$ with the boundary conditions $\x^*(0)=\x_0$ and $\x^*(t^*_f) \in\mathcal{M}\cap\mathcal{N}\backslash\{\bar{\x}(\bar{t}_f)\}$. Note that the domain $\Pi(\mathcal{F}_{\q})$  is not a $C^0$-topology neighborhood of the extremal trajectory $\bar{\x}(\cdot)$ on $[0,\bar{t}_f]$ since the initial state of each extremal $\Gamma(\cdot,\q)$ is the same. According to Agrachev's approach  in \cite{Agrachev:04} or Appendix A in \cite{Caillau:15}, one can construct a perturbed Lagrangian submanifold to prove that to be optimal in the domain $\Pi(\mathcal{F}_{\q})$ is sufficient for the strict strong-local optimality in $C^0$-topology.

Next, let us prove that Eq.~(\ref{EQ:lemma_1})  is a necessary condition as well. Note that, for every  $(t_f,\q)\in\mathcal{O}$, the extremal trajectory $\Pi(\Gamma(\cdot,\q))$ on $[0,t_f]$  is an admissible controlled trajectory  satisfying the boundary conditions  $\Pi(\Gamma(0,\q(\xi))) = \x_0$ and $\Pi(\Gamma(t_f,\q)) \in\mathcal{M}\cap\mathcal{N}$. Thus, if there exists a $(t_f,\q)\in\mathcal{O}$ such that Eq.~(\ref{EQ:lemma_1}) is not satisfied, the extremal trajectory $\bar{\x}(\cdot)$ on $[0,\bar{t}_f]$ is not locally optimal in the domain $\Pi(\mathcal{F}_{\q})$ any more, which proves the proposition.
\end{proof}

\noindent According to Eq.~(\ref{EQ:Hamiltonian}), we can rewrite $J(\xi)$ in Eq.~(\ref{EQ:definition_J_xi}) as
\begin{eqnarray}
J(\xi) = \int_{0}^{\tau(\xi)}\big[\p(t,\q(\xi)){\dot{\x}}(t,\q(\xi)) - H(\x(t,\q(\xi)),\p(t,\q(\xi)))\big]dt.
\end{eqnarray}
Let us define the path $\xi\mapsto \boldsymbol{\lambda}(\xi)$ in $T^*_{\y(\xi)}\mathcal{X}$ in such a way that $(\y(\xi),\boldsymbol{\lambda}(\xi)) = \Gamma(\tau(\xi),\q(\xi))$ for $\xi\in[-\varepsilon,\varepsilon]$. Then, for every $\xi\in[-\varepsilon,\varepsilon]$, the four paths $(\x(0,\q(\cdot)),\p(0,\q(\cdot)))$ on $[0,\xi]$, $\Gamma(\cdot,\q(\xi))$ on $[0,\tau(\xi)]$, $\Gamma(\cdot,\q(0))$ on $[0,\tau(0)]$, and $(\y(\cdot),\boldsymbol{\lambda}(\cdot))$ on $[0,\xi]$ form a closed curve on the family $\mathcal{F}_{\q}$. Since the integrant of the Poincar\'e-Cartan form $\p d\x - Hdt$ is exact on the family $\mathcal{F}_{\q}$ (cf. Proposition 17.2 in \cite{Agrachev:04}), it follows that
\begin{eqnarray}
&&J(0) +  \int_{0}^{\xi}\boldsymbol{\lambda}(\eta){\y^{\prime}}(\eta) - H(\y(\eta),\boldsymbol{\lambda}(\eta)){\tau^{\prime}}(\eta)d\eta \nonumber\\
&=& J(\xi) +  \int_{0}^{\xi}{\p_0}(0,\q(\eta)){\frac{d\x}{d\eta}}(0,\q(\eta)) - H(\x(0,\q(\eta)),\p(0,\q(\eta)))\frac{d t_0}{d\eta}d\eta,\nonumber
\end{eqnarray}
where the superscript ``~$\prime$~" denotes the derivative with respect to $\xi$. Since $\x(0,\q(\xi)) = \x_0$ and $t_0 = 0$, it follows that $J(\xi)$ can be further rewritten as
\begin{eqnarray}
 J(\xi)= J(0) + \int_{0}^{\xi}\boldsymbol{\lambda}(\eta){\y^{\prime}}(\eta) - H(\y(\eta),\boldsymbol{\lambda}(\eta)){\tau^{\prime}}(\eta)d\eta.\nonumber
\end{eqnarray} 
Note that $H(\x,\p) = 0$ for every point $(\x,\p)\in\mathcal{F}_{\q}$, we obtain
\begin{eqnarray}
J(\xi) = J(0) + \int_{0}^{\xi}\boldsymbol{\lambda}(\eta){\y^{\prime}}(\eta) d\eta.
\end{eqnarray} 
Hence, taking derivative of $J(\xi)$ with respect to $\xi$ leads to 
\begin{eqnarray}
J^{\prime}(\xi) = \boldsymbol{\lambda}(\xi)\y^{\prime}(\xi), \ \xi\in[-\varepsilon,\varepsilon].
\label{EQ:dJdxi}
\end{eqnarray}
Note that $\boldsymbol{\lambda}(0) = \bar{\p}(\bar{t}_f)$. According to the transversality condition in Eq.~(\ref{EQ:Transversality}), one has $\boldsymbol{\lambda}(0) \perp \y^{\prime}(0)$. Hence, the equation $J^{\prime}(0) = 0$ is satisfied for every smooth curve $\y(\cdot)\in\mathcal{M}\cap\mathcal{N}$ on $[-\varepsilon,\varepsilon]$. According to Proposition \ref{PR:compare}, we immediately obtain the following result.
\begin{corollary}
Given the extremal $(\bar{\x}(\cdot),\bar{\p}(\cdot)) = \Gamma(\cdot,\bar{\q})$ on $[0,\bar{t}_f]$ such that  {\it Assumptions} \ref{AS:Regular_Hamiltonian} and \ref{AS:Regular_Switching} as well as Conditions \ref{AS:Disconjugacy_bang} and \ref{AS:Transversality} are satisfied, let the subset $\mathcal{O}$ be small enough. Then, if $\varepsilon > 0$ is small enough, for every smooth curve $\y(\cdot)\in\mathcal{M}\cap\mathcal{N}$ on $[-\varepsilon,\varepsilon]$,
\begin{description}
\item $1)$ the strict inequality $J^{\prime\prime}(0) > 0$ is sufficient to ensure the extremal trajectory $\bar{\x}(\cdot)$ on $[0,\bar{t}_f]$ to be a strict strong-local optimum; and
\item $2)$ the inequality $J^{\prime\prime}(0) \geq 0$ is a necessary condition for the strict strong-local optimality of the extremal trajectory $\bar{\x}(\cdot)$ on $[0,\bar{t}_f]$.
\end{description}
\label{CO:necessary_sufficient}
\end{corollary}
\noindent Up to now, the inequality $J^{\prime\prime}(0) > 0$ still cannot be straightforwardly verified. In next paragraph, the numerically verifiable condition for $J^{\prime\prime}(0) > 0$ will be established.

Directly differentiating Eq.~(\ref{EQ:dJdxi}) with respect to $\xi$ yields
\begin{eqnarray}
J^{\prime\prime}(\xi) =  { \boldsymbol{\lambda}^{\prime}}(\xi){\y^{\prime}}(\xi) +  \boldsymbol{\lambda}(\xi){\y^{\prime\prime}}(\xi).
\label{EQ:d2Jdxi2}
\end{eqnarray}
Then, differentiating $\phi(\y(\xi)) = 0$ with respect to $\xi$ yields
\begin{eqnarray}
\frac{d  }{d\xi}\phi(\y(\xi))  &=&{{ \nabla \phi(\y(\xi))}} {\y^{\prime}(\xi)} = 0,\label{EQ:dphidxi}\nonumber\\
\frac{d^2  }{d\xi^2}\phi(\y(\xi))  &=& 
\left[\begin{array}{c}
\left(\y^{\prime}(\xi)\right)^T{\nabla^2 \phi_1(\y(\xi))}{\y^{\prime}(\xi)}\\
\vdots\\
\left({\y^{\prime}(\xi)}\right)^T{\nabla^2 \phi_s(\y(\xi))}{\y^{\prime}(\xi)}
\end{array}\right]
%\frac{d\y^T(\xi)}{d\xi}\frac{ d^2\phi(\y(\xi))}{d\y^Td\y}\frac{d\y(\xi)}{d\xi}
 + {\nabla \phi(\y(\xi))}{\y^{\prime\prime}(\xi)}= 0,
\label{EQ:dphi2dxi2}
\end{eqnarray}
where $\phi_i(\y(\xi))\in\mathbb{R}$ for $i=1,2,\cdots,s$ are the elements of the vector-valued function $\phi(\cdot)$. Let $\bar{\boldsymbol{\nu}}=[\bar{\nu}_1,\bar{\nu}_2,\cdots,\bar{\nu}_s]\in(\mathbb{R}^s)^*$ be the vector of the Lagrangian multipliers of the reference extremal $(\bar{\x}(\cdot),\bar{\p}(\cdot)) = \Gamma(\cdot,\bar{\q})$ on $[0,\bar{t}_f]$, i.e., $\bar{\p}(\bar{t}_f) =\bar{\boldsymbol{\nu}}\nabla\phi(\bar{\x}(\bar{t}_f))$; we immediately get $\boldsymbol{\lambda}(0) = \bar{\boldsymbol{\nu}}{\nabla \phi(\y(0))}$ because $\boldsymbol{\lambda}(0) = \bar{\p}(\bar{t}_f)$ and $\y(0) = \bar{\x}(\bar{t}_f)$. Multiplying $\bar{\boldsymbol{\nu}}$ on both sides of Eq.~(\ref{EQ:dphi2dxi2}) leads to
\begin{eqnarray}
\bar{\boldsymbol{\nu}}\frac{d^2  }{d\xi^2}\phi(\y(0))  &=&  \boldsymbol{\lambda}(0){\y^{\prime\prime}(0)} +
%\frac{d\y^T(0)}{d\xi} \bar{\boldsymbol{\nu}}^T\frac{ d^2\phi(\y(0))}{d\y^Td\y}\frac{d\y(0)}{d\xi}
\sum_{i=1}^{s}\bar{\nu}_i[{\y^{\prime}(0)}]^T{\nabla^2 \phi_i(\y(0))}{\y^{\prime}(0)}
\nonumber\\
% &=&  \boldsymbol{\lambda}^T(0)\frac{d^2\y(0)}{d\xi^2} +
%\sum_{i=1}^{l}\frac{d\y^T(0)}{d\xi} \left[\bar{\nu}_i \frac{d^2 \phi_i(\y(0))}{d\y^Td\y}\right]\frac{d\y(0)}{d\xi}
%\nonumber\\
&=&\boldsymbol{\lambda}(0){\y^{\prime\prime}(0)}  +[{\y^{\prime}(0)}]^T
\left[\sum_{i=1}^{s} \bar{\nu}_i {\nabla^2 \phi_i(\y(0))}\right]{\y^{\prime}(0)} = 0.
\nonumber
\end{eqnarray}
Substituting this equation into Eq.~(\ref{EQ:d2Jdxi2}) yields
\begin{eqnarray}
{ J^{\prime\prime}}(0) ={\boldsymbol{\lambda}^{\prime}(0)} {\y^{\prime}(0)}-[{\y^{\prime}(0)}]^T
\left[\sum_{i=1}^{s} \bar{\nu}_i {\nabla^2 \phi_i(\y(0))}\right]{\y^{\prime}(0)}.
\label{EQ:d2Jdxi20}
\end{eqnarray}
Note that $\y(\xi) = \x(\tau(\xi),\q(\xi))$ and $\boldsymbol{\lambda}(\xi) = \p(\tau(\xi),\q(\xi))$ for every $\xi\in[-\varepsilon,\varepsilon]$. Then, taking their derivatives with respect to $\xi$ leads to 
\begin{eqnarray}
\y^{\prime}(\xi) &=& {\nabla \x(\tau(\xi),\q(\xi))}{(\tau^{\prime}(\xi),\q^{\prime}(\xi))^T},\label{EQ:dydxi}\\
\left[\boldsymbol{\lambda}^{\prime}(\xi)\right]^T &=& {\nabla \p^T(\tau(\xi),\q(\xi))} {\left(\tau^{\prime }(\xi),\q^{\prime}(\xi)\right)^T}.\label{EQ:dlambdadxi}
\end{eqnarray}
Substituting Eq.~(\ref{EQ:dydxi}) and Eq.~(\ref{EQ:dlambdadxi}) into Eq.~(\ref{EQ:d2Jdxi20}) and considering that the matrix ${\nabla \x(\tau(\xi),\q(\xi))}$ is nonsingular under Condition \ref{AS:Disconjugacy_bang}, we obtain that the equation
\begin{eqnarray}
{ J^{\prime\prime}}(0) = \left[{\y^{\prime}(0)}\right]^T\Big\{{\nabla \p^T(\bar{t}_f,\bar{\q})}\left[{\nabla \x(\bar{t}_f,\bar{\q})} \right]^{-1} - \sum_{i=1}^{s} \bar{\nu}_i {\nabla^2 \phi_i(\y(0))} \Big\}{\y^{\prime}(0)},
\label{EQ:d2Jdxi201}
\end{eqnarray}
is satisfied for every smooth curve $\y(\cdot)\in\mathcal{M}\cap\mathcal{N}$ on $[-\varepsilon,\varepsilon]$.
\begin{definition}
Let $\boldsymbol{T}\in\mathbb{R}^{(n)\times (n-s)}$ be a full rank matrix such that each of its column vector is orthogonal to the normal vector of the submanifold $\mathcal{M}$ at $\bar{\x}(t_f)$, i.e., ${\nabla  \phi(\bar{\x}(t_f))}\boldsymbol{T} = 0$.
\label{DE:C}
\end{definition}
\noindent Note that the column vectors of $\boldsymbol{T}$ form a basis of the tangent space $T_{\bar{\x}(\bar{t}_f)}\mathcal{M}$. Since the vector ${\y^{\prime}(0)}$ is tangent to the manifold $\mathcal{M}$ at $\bar{\x}(\bar{t}_f)$, for every curve $\y(\cdot)\in\mathcal{M}\cap\mathcal{N}$ on $[-\varepsilon,\varepsilon]$, there exists a vector $\boldsymbol{\zeta}\in\mathbb{R}^{n-s}$ such that ${\y^{\prime}(0)}= \boldsymbol{T}\boldsymbol{\zeta}$. Then, substituting ${\y^{\prime}(0)}= \boldsymbol{T}\boldsymbol{\zeta}$ into Eq.~(\ref{EQ:d2Jdxi201}), we obtain
\begin{eqnarray}
J^{\prime\prime}(0) = \boldsymbol{\zeta}^T\boldsymbol{T}^T\Big\{{\nabla \p^T(\bar{t}_f,\bar{\q})}\big[{\nabla \x(\bar{t}_f,\bar{\q})} \big]^{-1} - \sum_{i=1}^{s} \bar{\nu}_i {\nabla^2 \phi_i(\y(0))} \Big\}\boldsymbol{T}\boldsymbol{\zeta}.\nonumber
%\label{EQ:d2Jdxi21}
\end{eqnarray}
Since the vector $\boldsymbol{\zeta}$ can take arbitrary values in $\mathbb{R}^{n-s}$, it follows that the strict inequality $J^{\prime\prime}(0)  > 0$ is satisfied if and only if there holds
\begin{eqnarray}
\boldsymbol{T}^T\Big\{{\nabla \p^T(\bar{t}_f,\bar{\q})}\left[{\nabla \x(\bar{t}_f,\bar{\q})} \right]^{-1} - \sum_{i=1}^{s} \bar{\nu}_i {\nabla^2 \phi_i(\y(0))} \Big\}\boldsymbol{T} \succ 0.
 \label{EQ:numerical_second_transversality}
\end{eqnarray}
This equation generalizes the second-order condition for fixed-time problems in \cite{Chen:153bp}  to the problems with free final time.  
\begin{condition}
Let Eq.~(\ref{EQ:numerical_second_transversality}) be satisfied at the final point of the reference  extremal $\Gamma(\cdot,\bar{\q})$ on $[0,\bar{t}_f]$.
\label{AS:terminal_condition}
\end{condition}
\noindent Then, as a result of Corollary \ref{CO:necessary_sufficient}, we eventually obtain the following theorem.
\begin{theorem}
Given the extremal $(\bar{\x}(\cdot),\bar{\p}(\cdot)) = \Gamma(\cdot,\bar{\q})$ on $[0,\bar{t}_f]$, let {\it Assumptions} \ref{AS:Regular_Hamiltonian} and \ref{AS:Regular_Switching} be satisfied. Then, if Conditions \ref{AS:Disconjugacy_bang}, \ref{AS:Transversality}, and \ref{AS:terminal_condition} are satisfied, the extremal trajectory $\bar{\x}(\cdot)$ on $[0,\bar{t}_f]$ realizes a strict strong-local optimum.\label{CO:cor3}
\end{theorem}
\noindent Accordingly, in the case of $s<n$, if the regularity conditions in  {\it Assumptions} \ref{AS:Regular_Hamiltonian} and \ref{AS:Regular_Switching} are satisfied, {\it Conditions} \ref{AS:Disconjugacy_bang}, \ref{AS:Transversality}, and \ref{AS:terminal_condition}  are sufficient to guarantee the reference extremal to be a strict strong-local optimum. In next section, a numerical implementation for {\it Conditions} \ref{AS:Disconjugacy_bang}, \ref{AS:Transversality}, and \ref{AS:terminal_condition} is derived.

\section{Numerical implementation}\label{SE:Procedure}

In this section, we assume that the reference extremal $(\bar{\x}(\cdot),\bar{\p}(\cdot))=\Gamma(\cdot,\bar{\q})$ on $[0,\bar{t}_f]$ is computed by the PMP. Then, one can  directly check the regularity conditions in Assumptions \ref{AS:Regular_Hamiltonian} and \ref{AS:Regular_Switching}. 

Once the explicit expression of the function $\phi(\x)$ is given, one can manually derive the two matrices ${\nabla \phi(\bar{\x}(\bar{t}_f))}$ and ${\nabla^2 \phi_i(\bar{\x}(\bar{t}_f))}$ for $i=1,2,\cdots,s$.  Note that the matrix $\boldsymbol{T}$ in Definition \ref{DE:C} can be computed by a simple Gram-Schmidt process if the matrix ${\nabla \phi(\bar{\x}(\bar{t}_f))}$ is derived. According to Eq.~(\ref{EQ:Transversality}), we have
\begin{eqnarray}
\bar{\boldsymbol{\nu}} =\bar{\p}(\bar{t}_f){\nabla \phi^T(\bar{\x}(\bar{t}_f))}  \left[{\nabla \phi(\bar{\x}(\bar{t}_f))} \nabla \phi^T(\bar{\x}(\bar{t}_f))\right]^{-1}.
\label{EQ:compute_nu}
\end{eqnarray}
Therefore, with the exception of the two matrices ${\nabla \x(\cdot,\bar{\q})}$ and ${\nabla \p^T(\cdot,\bar{\q})}$ on $[0,\bar{t}_f]$, all the necessary quantities for testing {\it Conditions} \ref{AS:Disconjugacy_bang}, \ref{AS:Transversality}, and \ref{AS:terminal_condition} are available.

The two vectors ${\dot{\x}(t,\bar{\q})}$ and ${\dot{\p}(t,\bar{\q})}$ can be immediately obtained once the  extremal $\Gamma(\cdot,\bar{\q})$ on $[0,\bar{t}_f]$ is given. It follows from the classical results about solutions to ODEs that the trajectory $(\x(t,\bar{\q}),\p(t,\bar{\q}))$ and its time derivative $(\dot{\x}(t,\bar{\q}),\dot{\p}(t,\bar{\q}))$ are continuously differentiable with respect to $\q$ on each subinterval $(\bar{t}_i,\bar{t}_{i+1})$ for $i=0,1,\cdots,k$. Hence, differentiating Eq.~(\ref{EQ:cannonical}) with respect to $\q$ leads to
% Note that $\frac{d \rho}{d\q^T} = 0$ since $\rho$ is a piecewise constant on nonsingular extremals. Then, differentiating Eq.~(\ref{EQ:cannonical}) leads to
\begin{eqnarray}
\frac{d}{dt}\frac{\partial \x}{\partial \q}(\cdot,\bar{\q}) &=& H_{\p\x}(\bar{\x}(\cdot),\bar{\p}(\cdot))\frac{\partial \x}{\partial \q} (\cdot,\bar{\q}) +H_{\p\p}(\bar{\x}(\cdot),\bar{\p}(\cdot))\frac{\partial \p^T}{\partial \q}(\cdot,\bar{\q}) ,\label{EQ:dt_dx_dq}\\
\frac{d}{dt}\frac{\partial \p^T}{\partial \q} (\cdot,\bar{\q}) &=& -H_{\x\x}(\bar{\x}(\cdot),\bar{\p}(\cdot))\frac{\partial \x}{\partial \q}(\cdot,\bar{\q})  - H_{\x\p}(\bar{\x}(\cdot),\bar{\p}(\cdot))\frac{\partial \p^T}{\partial \q}(\cdot,\bar{\q}),\label{EQ:dt_dp_dq}
\end{eqnarray}
on $(\bar{t}_i,\bar{t}_{i+1})$. Since the initial point $\x_0$ is fixed, we obtain
\begin{eqnarray}
\frac{\partial \x(0,\bar{\q})}{\partial \q} = \frac{d\x_0}{d\q} = \boldsymbol{0}_{n\times(n-1)}.
\label{EQ:initial_condition_x}
\end{eqnarray}
 The initial value of ${\partial \p^T(t,\bar{\q})}/{\partial \q}$ can be obtained by
\begin{eqnarray}
\frac{\partial \p^T(0,\bar{\q})}{\partial \q} = \frac{d \left(\bar{\p}_0 + \q \boldsymbol{E}^T\right)^T}{d\q} =  \boldsymbol{E},
\label{EQ:initial_condition_p}
\end{eqnarray}
where the matrix $\boldsymbol{E}$ can be computed by employing a simple Gram-Schmidt process once the vector $\frac{\partial H}{\partial \p^T}(\x_0,\bar{\p}_0)=\f(\x_0,\u(\x_0,\bar{\p}_0))$ is given. 
Note that the analytical solution to the state transition matrix $\Psi(t,\bar{t}_i)$ on coast arcs was  derived by Glandorf in \cite{Glandorf:69}. Thus, one can use
\begin{eqnarray}
\frac{\partial \x}{\partial \q}(t,\bar{\q}) = \Psi(t,\bar{t}_i)\frac{\partial \x}{\partial \q}(\bar{t}_i,\bar{\q}),\ t\in(\bar{t}_i,\bar{t}_{i+1}),
\end{eqnarray} 
to avoid numerical integration on coast arcs. Since the transition matrix $\Psi(t,\bar{t}_i)$ on coast arcs is nonsingular, i.e., $\det\left[\Psi(t,\bar{t}_i)\right] \neq 0$, it follows that there exist no conjugate points on coast arcs if the starting point of the coast arcs is not a conjugate one.

The matrices ${\partial \x}(t,\bar{\q})/{\partial \q}$ and ${\partial \boldsymbol{p}}(t,\bar{\q})/{\partial \q}$ are discontinuous at the each switching time $\bar{t}_i$ for $i=1,2,\cdots,k$. By comparing with the development in \cite{Noble:02}, we obtain that the updating formulas for the two matrices at each switching time $\bar{t}_i$ are 
\begin{eqnarray}
\begin{cases}
\frac{\partial \x}{\partial \q}(\bar{t}_i+,\bar{\q})= \frac{\partial \x}{\partial \q}(\bar{t}_i-,\bar{\q})  - \Delta \rho_i \f_1(\bar{\x}(\bar{t}_i),\bar{\boldsymbol{\omega}}(\bar{t}_i))\frac{ d t_i(\bar{\q})}{d \q},\\%\label{EQ:update_formula_x}\\
\frac{\partial \p^T}{\partial \q}(\bar{t}_i+,\bar{\q}) = \frac{\partial \p^T}{\partial \q}(\bar{t}_i-,\bar{\q}) + \Delta \rho_i \frac{\partial  \f_1}{\partial  \x^T}(\bar{\x}(\bar{t}_i),\bar{\boldsymbol{\omega}}(\bar{t}_i))\bar{\p}^T(\bar{t}_i)\frac{d t_i(\bar{\q})}{d\q},
\end{cases}
\label{EQ:update_formula}
\end{eqnarray}
where $\Delta \rho_i = \bar{\rho}(\bar{t}_i+) - \bar{\rho}(\bar{t}_i -)$. Up to now, with the exception of ${d t_i(\bar{\q})}/{d \q}$, all the necessary quantities can be computed. Differentiating $H_1(\x(t_i(\q),\q),\p(t_i(\q),\q)) = 0$ with respect to $\q$, one gets
\begin{eqnarray}
\dot{H}_1(\bar{\x}(\bar{t}_i),\bar{\p}(\bar{t}_i)) \frac{d t_i(\bar{\q})}{d \q} + \frac{\partial H_1(\bar{\x}(\bar{t}_i),\bar{\p}(\bar{t}_i))}{\partial \x^T}\frac{\partial \x}{\partial \q}(\bar{t}_i,\bar{\q}) + \frac{\partial H_1(\bar{\x}(\bar{t}_i),\bar{\p}(\bar{t}_i)) }{\partial \p}\frac{\partial \p^T}{\partial \q}(\bar{t}_i,\bar{\q}) = 0.\nonumber
\end{eqnarray}
Since $\dot{H}_1(\bar{\x}(\bar{t}_i),\bar{\p}(\bar{t}_i)) \neq 0$ by {\it Assumption} \ref{AS:Regular_Switching}, we eventually obtain
\begin{eqnarray}
\frac{d t_i(\bar{\q})}{d \q}  = -\left[\frac{\partial H_1(\bar{\x}(\bar{t}_i),\bar{\p}(\bar{t}_i))}{\partial \x^T}\frac{\partial \x}{\partial \q}(\bar{t}_i,\bar{\q}) +\frac{\partial H_1(\bar{\x}(\bar{t}_i),\bar{\p}(\bar{t}_i)) }{\partial \p}\frac{\partial \p^T}{\partial \q}(\bar{t}_i,\bar{\q})\right]/\dot{H}_1(\bar{\x}(\bar{t}_i),\bar{\p}(\bar{t}_i)).\nonumber
%\label{EQ:dti_dq}
\end{eqnarray}

Therefore, in order to compute the two matrices ${\partial \x}(\cdot,\bar{\q})/{\partial \q}$ and ${\partial \boldsymbol{p}^T}(\cdot,\bar{\q})/{\partial \q}$ on $[0,\bar{t}_f]$,  it amounts to choose the initial conditions in Eq.~(\ref{EQ:initial_condition_x}) and Eq.~(\ref{EQ:initial_condition_p}), then to numerically integrate  the homogeneous linear differential equations in Eq.(\ref{EQ:dt_dx_dq}) and Eq.~(\ref{EQ:dt_dp_dq}) on each smooth bang arc while using the updating formulas in Eq.~(\ref{EQ:update_formula}) once a switching point is encountered.  

\section{Numerical Examples}\label{SE:Numerical}

A typical orbital transfer from an inclined geosynchronous transfer orbit to the geostationary one is considered.  The {\it modified equinoctial orbital elements} (MEOE) developed by Broucke and Cefola \cite{Broucke:72} are used for numerical computations. The MEOE describe the orbit by the semilatus rectum $P\in\mathbb{R}$, the eccentricity vector $(e_x,e_y)\in\mathbb{R}^2$, the inclination vector $(h_x,h_y)\in\mathbb{R}^2$, and the true longitude $l\in\mathbb{R}$.  The values of $P$, $e_x$, $e_y$, $h_x$, and $h_y$ for the initial and final orbits  are presented in Tab. $\ref{Tab:Parameters}$.
 \begin{table}[!ht]
    \renewcommand{\arraystretch}{1}
     \caption[]{The values of $P$, $e_x$, $e_y$, $h_x$, and $h_y$ for initial and final orbits.}
    \centering
    \begin{tabular}{ccc}
    \hline\hline
    {\it MEOE} & 
   {Initial orbit}
    &{Final orbit}\\
    \hline 
    $P$     &  $11,625.00$ km                  & $42,165.00$ km        \\%\hline
    $e_x$     & 0.75              & 0     \\%\hline
    $e_y$     & 0        & 0      \\%\hline
    $h_x$     & 6.12$\times$10$^{-2}$       & 0   \\%\hline
    $h_y$     & 0          & 0      \\ \hline
    \end{tabular}
\label{Tab:Parameters}
\end{table}
The Earth gravitational constant $\mu$ in Eq.~(\ref{EQ:Sigma}) is $398600.47$ km$^{3}$/s$^{2}$. The initial mass $m_0$ of the spacecraft is $1500$ kg and the specific impulse for the engine is $I_{sp} = 2000$ s. Since $\beta = 1/(I_{sp}g_0)$ where $g_0 = 9.8$ m/s$^2$, we obtain $\beta = 5.1 \times 10^{-5}$ s/m. The initial true longitude is fixed as $\pi$, i.e., $l_0 = \pi$ rad. We consider two cases (case A and case B) of orbital transfers with different value of $u_{max}$ and different final true longitude.

\subsection{Case A}\label{Subsection:first_example}

Let $l_f =9\times 2\pi$ rad and $u_{max} = 10$ N for case A. The homotopy method proposed in  \cite{Gergaud:06} is employed to compute the candidate solution (or the reference extremal). It is worth remarking that, since abnormal extremals do not exist for the fuel-optimal problem, the homotopy method converges if no conjugate points occur \cite{Trelat:12}.  The computed final time is $\bar{t}_f\approx 146.36$ h. The 3-dimensional profile of the position vector $\r(\cdot)$ on $[0,\bar{t}_f]$ and its projections onto $xy$- and $yz$-planes are plotted in Fig. \ref{Fig:Transferring_Orbit3},
\begin{figure}[!ht]
 \centering\includegraphics[trim=0cm 0cm 0cm 0cm, clip=true, width=4in, angle=0]{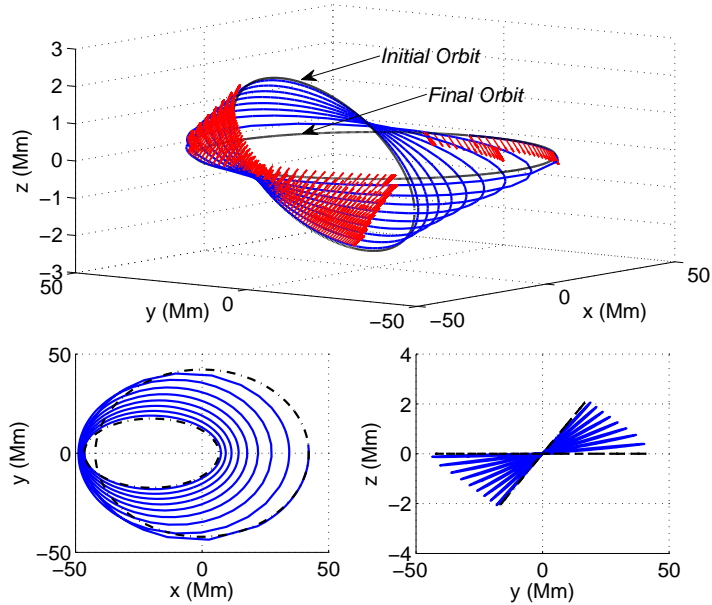}
 \caption[]{The 3-dimensional profile of the position vector $\r$ for case A and the arrows denote the thrust direction on burn arcs. The left and right bottom plots are the projections of $\r$ onto $xy$- and $yz$-plane, respectively.}
 \label{Fig:Transferring_Orbit3}
\end{figure}
and the time histories of $e_x$, $e_y$, $h_x$, and $h_y$ are demonstrated in Fig. \ref{Fig:exeyhxhy1}.
\begin{figure}[!ht]
 \centering\includegraphics[trim=1cm 0cm 0cm 1cm, clip=true, width=4in, angle=0]{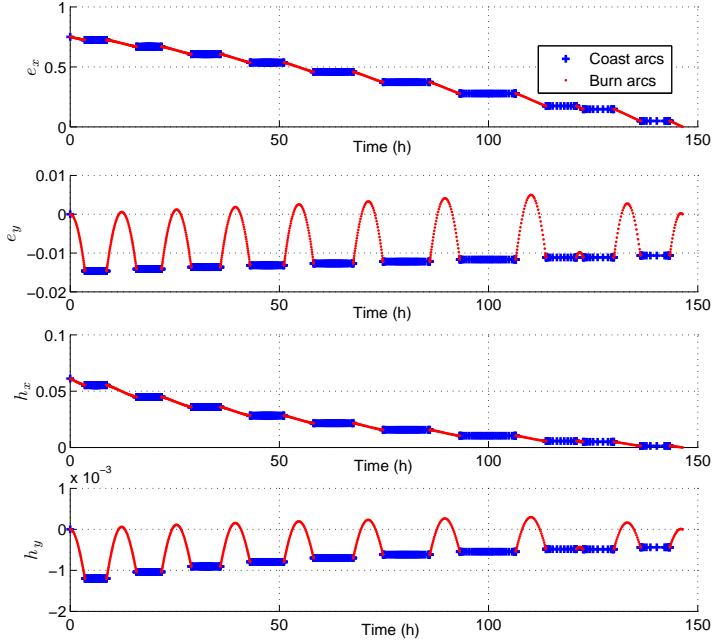}
 \caption[]{The time histories of $e_x$, $e_y$, $h_x$, and $h_y$ for case A.}
 \label{Fig:exeyhxhy1}
\end{figure}
 It is apparent that the number of burn arcs for case A is 11 with 20 switching points.
\begin{figure}[!ht]
 \centering\includegraphics[trim=1cm 0cm 0.0cm 0cm, clip=true, width=4in, angle=0]{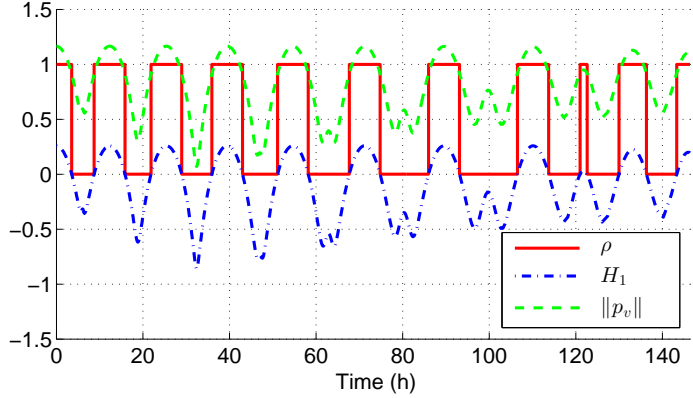}
 \caption[]{The time histories of $\rho$, $H_1$, and $\parallel \p_v\parallel$ for case A.}
 \label{Fig:Transferring_Orbit4}
\end{figure}
We can also see from Fig. \ref{Fig:Transferring_Orbit4} that each switching point is regular (cf. {\it Assumptions} \ref{AS:Regular_Switching}). Then, directly applying the numerical procedure in Sect. \ref{SE:Procedure}, one can compute the piecewise continuous function $\delta(\cdot)$ on $[0,\bar{t}_f]$. In order to have a clear view, Fig. \ref{Fig:det_4xtf}
\begin{figure}[!ht]
 \centering\includegraphics[trim=1.0cm 0.5cm 1.5cm 0cm, clip=true, width=4in, angle=0]{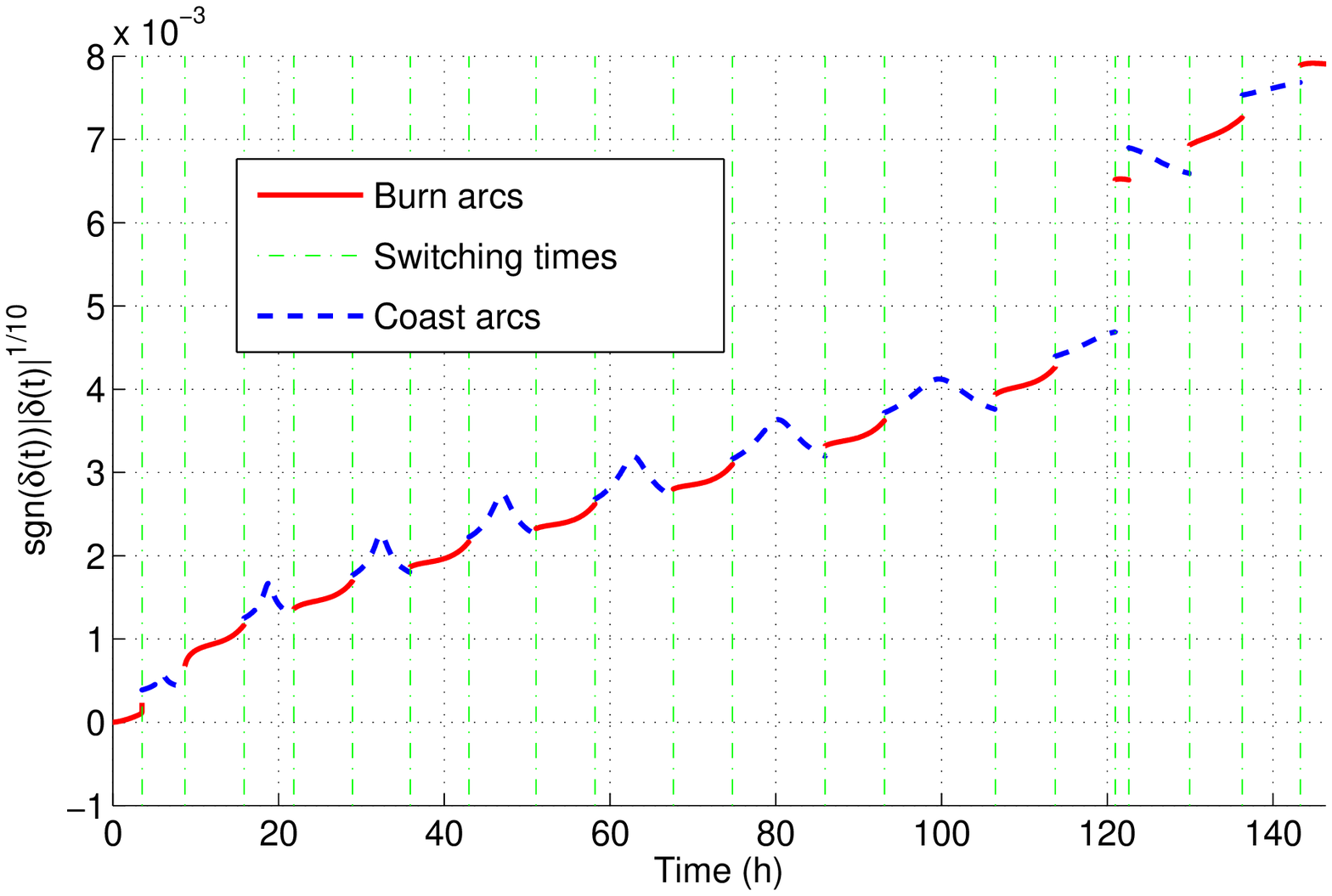}
 \caption[]{The time histories of ${sgn}(\delta(t))|\delta(t)|^{1/10}$ for case A.}
 \label{Fig:det_4xtf}
\end{figure}
shows instead the profile of $\text{sgn}(\delta(\cdot))\times |\delta (\cdot)|^{1/10}$, which can capture the sign property of $\delta(\cdot)$ on $[0,\bar{t}_f]$. We can see that there exist no zeros on the semi-open interval $(0,\bar{t}_f]$ and no sign change at each switching point, i.e., {\it Conditions} \ref{AS:Disconjugacy_bang} and \ref{AS:Transversality} are satisfied.  

Since the final point is not fixed, we have to check  {\it Condition} \ref{AS:terminal_condition}. Note that only the final mass $m_f$ is left free. We thus obtain  $s = 6$ and
   ${\nabla \phi(\bar{\x}(\bar{t}_f))}= \big[
I_{n-1} \ \ \boldsymbol{0}_{(n-1)\times 1} \big],$
which implies $\boldsymbol{T} = [
\boldsymbol{0}_{1\times 6} \ \ 1]^T$ and ${\nabla^2 \phi_i(\bar{\x}(\bar{t}_f))}= 0$ for $i=1,2,\cdots,s$. 
Substituting these values into Eq.~(\ref{EQ:numerical_second_transversality}) and Eq.~(\ref{EQ:compute_nu}), we obtain 
\begin{eqnarray}
 \boldsymbol{T}^T\Big\{{\nabla \p^T(\bar{t}_f,\bar{\q})}\left[{\nabla \x(\bar{t}_f,\bar{\q})} \right]^{-1} - \sum_{i=1}^{s} \bar{\nu}_i {\nabla^2 \phi_i(\bar{\x}(\bar{t}_f))} \Big\}\boldsymbol{T} \approx 4.6186\times 10^{11} \succ 0,\nonumber
 \end{eqnarray}
 which indicates that {\it Condition} \ref{AS:terminal_condition} is met.
Therefore, the computed extremal trajectory for case A realizes a strict strong-local optimum according to {\it Theorem} \ref{CO:cor3}. 

\subsection{Case B}

For case B, let $l_f =  19\times2\pi$ rad and we consider a lower value of $u_{{max}}$, i.e., $u_{max} = 5$ N. The optimal candidate solution is computed and shown in Fig. \ref{Fig:Transferring_Orbit}.
\begin{figure}[!ht]
 \centering\includegraphics[trim=0cm 0.0cm 0.0cm 0cm, clip=true, width=4in, angle=0]{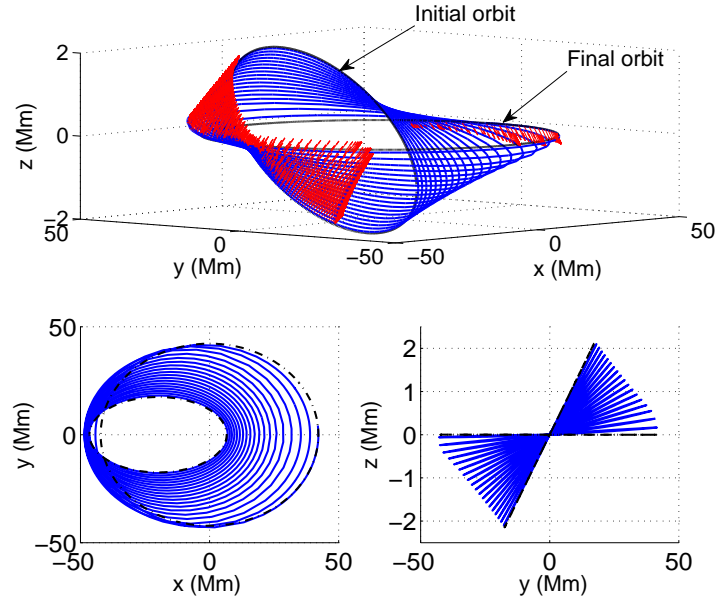}
 \caption[]{The 3-dimensional profile of $\r$ for case B and the arrows denote the thrust direction on burn arcs. The left and right bottom plots are the projections of $\r$ onto $xy$- and $yz$-planes, respectively.}
 \label{Fig:Transferring_Orbit}
\end{figure}
The computed transfer time is  $\bar{t}_f \approx 316.38$ h. The profiles of $e_x$, $e_y$, $h_x$, and $h_y$ against time are plotted in Fig. \ref{Fig:exeyhxhy}.
\begin{figure}[!ht]
 \centering\includegraphics[trim=0.4cm 0.0cm 0.5cm 0cm, clip=true, width=4in, angle=0]{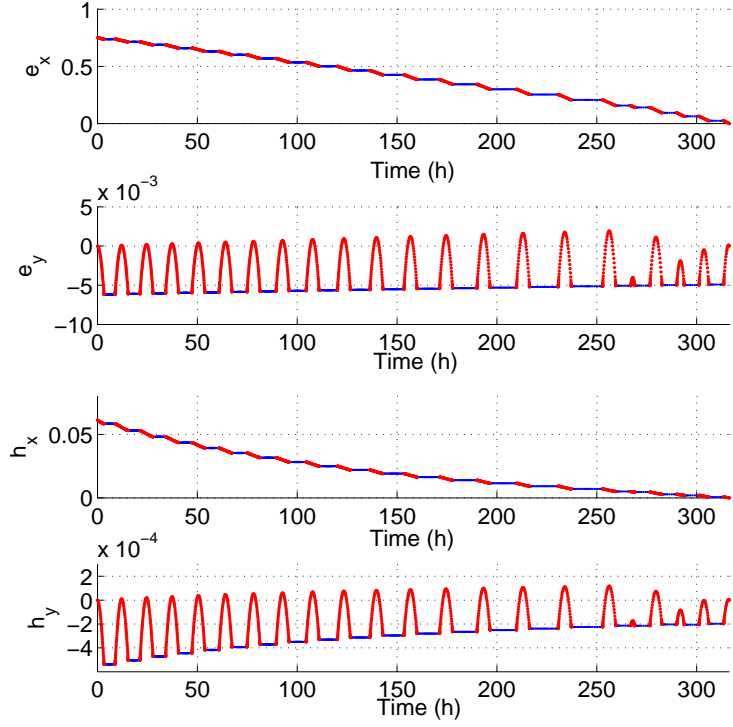}
 \caption[]{The time histories of $e_x$, $e_y$, $h_x$, and $h_y$ for case B.}
 \label{Fig:exeyhxhy}
\end{figure}
To see the regularity conditions, the time histories of $\rho$, $H_1$, and $\parallel \p_v \parallel$ are illustrated in Fig. \ref{Fig:control_h1}, showing that  {\it Assumption} \ref{AS:Regular_Switching} is met.
\begin{figure}[!ht]
 \centering\includegraphics[trim=1cm 0.0cm 1cm 0cm, clip=true, width=4in, angle=0]{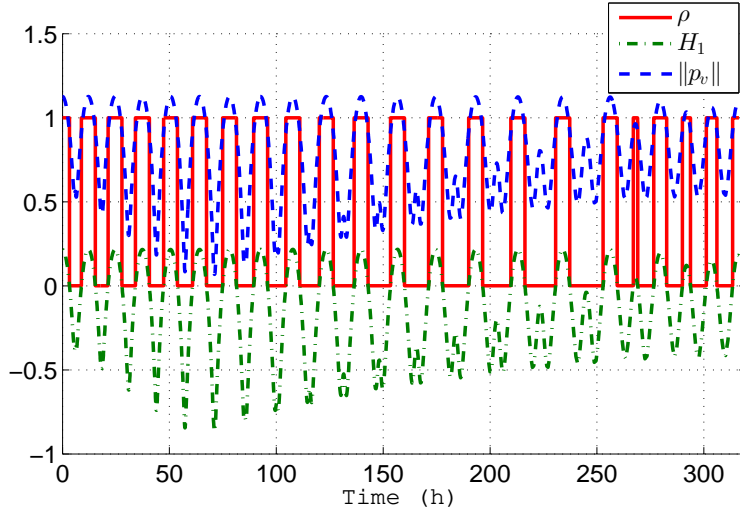}
 \caption[]{The time histories $\rho$, $\|\p_{v}\|$, and $H_1$ for case B.}
 \label{Fig:control_h1}
\end{figure}
By applying the numerical procedure in Sect. \ref{SE:Procedure}, the profile of $\text{sgn}(\delta(\cdot))|\delta(\cdot)|^{1/18}$ on $[0,\bar{t}_f]$ is computed and demonstrated in Fig. \ref{Fig:det}.
\begin{figure}[!ht]
 \centering\includegraphics[trim=0cm 0cm 0cm 0cm, clip=true, width=4in, angle=0]{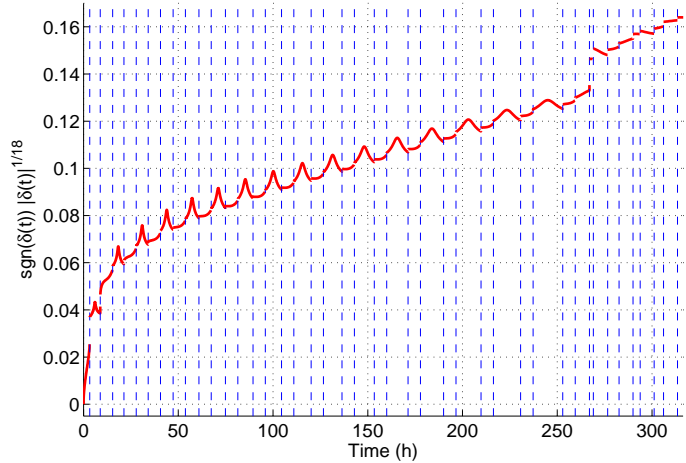}
 \caption[]{The profile of $sgn(\delta(t))|\delta(t)|^{1/18}$ for case B and the vertical dashed lines denote switching times.}
 \label{Fig:det}
\end{figure}
It is clear that there are no zeros on the semi-open interval $(0,\bar{t}_f]$ and no sign change at each switching time, i.e., {\it Conditions} \ref{AS:Disconjugacy_bang} and \ref{AS:Transversality} are met. The same as case A, we have
${\nabla \phi(\bar{\x}(\bar{t}_f))}= \big[
I_{n-1} \ \ \boldsymbol{0}_{(n-1)\times 1}\big],$
 $\boldsymbol{T} = [
\boldsymbol{0}_{1\times 6} \ \ 1]^T$, and ${\nabla^2\phi_i(\bar{\x}(\bar{t}_f))} = \boldsymbol{0}$ for $i=1,2,\cdots,s$. Thus, directly substituting the numerical values of $\bar{\x}(\bar{t}_f)$ and $\bar{\p}(\bar{t}_f)$ into Eq.~(\ref{EQ:numerical_second_transversality}) and Eq.~(\ref{EQ:compute_nu}), we obtain
\begin{eqnarray}
 \boldsymbol{T}^T\Big\{{\nabla \p^T(\bar{t}_f,\bar{\q})}\left[{\nabla \x(\bar{t}_f,\bar{\q})} \right]^{-1} - \sum_{i=1}^{s} \bar{\nu}_i {\nabla^2 \phi_i(\bar{\x}(\bar{t}_f))} \Big\}\boldsymbol{T} \approx 8.402\times 10^{9} \succ 0,\nonumber
 \end{eqnarray}
 which indicates that {\it Condition} \ref{AS:terminal_condition} is met.
Up to now, all the conditions in Theorem \ref{CO:cor3} are satisfied. Therefore, the computed trajectory for case B realizes a strict strong-local optimum.
 
To see the occurrence of conjugate points, the profile of $\text{sgn}(\delta(\cdot))|\delta(\cdot)|^{1/18}$ on the  time interval extended to $[0,1000]$ is demonstrated in Fig. \ref{Fig:det_extended}. 
\begin{figure}[!ht]
 \centering\includegraphics[trim=0cm 0cm 0cm 0cm, clip=true, width=4in, angle=0]{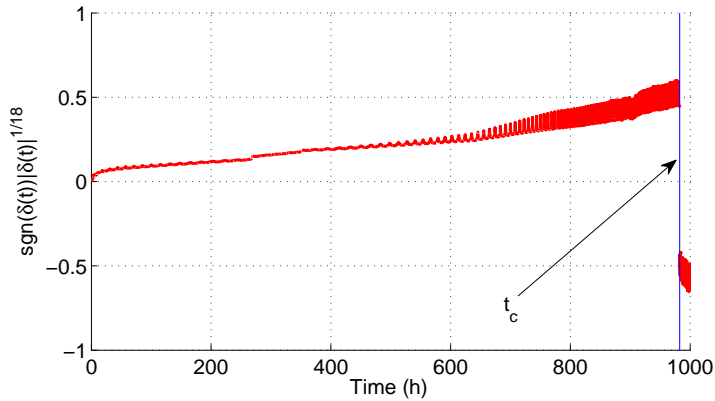}
 \caption[]{The profile of $sgn(\delta(t))|\delta(t)|^{1/18}$ for case B on the extended time interval $[0,1000]$.}
 \label{Fig:det_extended}
\end{figure}
Apparently, there is a sign change at a switching time $t_c\approx 982.63$  which violates {\it Condition} \ref{AS:Transversality}. Thus,  the trajectory $\bar{\x}(\cdot)$ on $[0,t_f]$ of case B is not optimal any more if $t_f > t_c$.  Note that the classical method of calculus of variations \cite{Bryson:69} to test the explosive time of the matrix ${\nabla  \p^T(t,\bar{\q})}\left[{\nabla \x(t,\bar{\q})}\right]^{-1}$ fails to find the conjugate time $t_c$ in Fig. \ref{Fig:det_extended}.

\section{Conclusions}
This paper is concerned with establishing the second-order necessary and sufficient optimality conditions as well as their numerical implementations for the free-time multi-burn orbital transfer problems. Through analyzing the projection behaviour of the parameterized family of extremals constructed in this paper, two no-fold conditions (cf. Conditions \ref{AS:Disconjugacy_bang} and \ref{AS:Transversality}) ensuring the projection of the parameterized family to be a diffeomorphism are established. As a result, it is obtained that conjugate points for the multi-burn problem may occur not only on burn arcs but also at switching times and that the absence of conjugate points  is sufficient to guarantee the reference extremal to be locally optimal if the final state is fixed. For the case that the final state is not fixed but varies on a smooth target manifold, an extra second-order necessary and sufficient condition, involving the geometry of the target manifold, is established. It is worth remarking that the development in this paper is applicable  not only to bang-bang extremals but also to totally smooth extremals, e.g., the extremals of time-optimal orbital transfers.   Finally, two fuel-optimal transfer trajectories are calculated, and the optimality conditions developed in this paper are tested to show that the two computed extremals are locally optimal.

\end{document}